\documentclass[a4paper,11pt,reqno]{amsart} 

\usepackage{amssymb, amsmath, amscd, amsthm}
\usepackage[color,all]{xy}
\usepackage{hyperref}
\usepackage{color}
\usepackage[T1]{fontenc}
\usepackage[utf8]{inputenc}
\usepackage{thmtools}
\usepackage{tikz-cd}
\usepackage{enumitem}

\newtheorem{theorem}{{\textbf Theorem}}[section]
\newtheorem{proposition}[theorem]{{\textbf Proposition}}
\newtheorem{corollary}[theorem]{{\textbf Corollary}}
\newtheorem{lemma}[theorem]{{\textbf Lemma}}
\newtheorem{exa}[theorem]{{\textbf Example}}
\newtheorem{remit}[theorem]{{\textbf Remark}}
\newtheorem{defn}[theorem]{{\textbf Definition}}

\newenvironment{remark}{\begin{remit}\rm}{\end{remit}}
\newenvironment{definition}{\begin{defn}\rm}{\end{defn}}

\numberwithin{equation}{section}
\title{The Segre Invariant and Brill-Noether theory on Hirzebruch Surfaces}
\author{Mauricio Rivera-Vega}
\address{Centro de Ciencias Matemáticas - UNAM, Morelia, Campus Morelia, Antigua Carretera a Pátzcuaro 8701, Col. Ex Hacienda San José de la Huerta, Morelia, Michoacán, México C. P. 58089.}
\email{frivera@matmor.unam.mx}
\begin{document}

\maketitle

\begin{abstract}
We study the Segre invariant on rank two $H$-stable vector bundles with fixed Chern classes on Hirzebruch surfaces $\Sigma_{n}$. For each vector bundle on these surfaces, we determine the possible values of the Segre invariant and compute the number of global sections of these bundles and present applications to Brill–Noether Theory.
\end{abstract}

\section{Introduction}
We work over the field $\mathbb{C}$ of complex numbers. Given coherent sheaves $F$ and $G$ on a variety $X$ we will write $h^{i}(F)$ (resp. $ext^{i}(G,F)$) to denote the dimension of the $i$th cohomology group $H^{i}(X,F)=H^{i}(F)$ (resp. $i$th $Ext$ group $Ext^{i}(G,F)$) as a $\mathbb{C}$-vector space. We will denote by $F^{*}$ the dual of $F$ and by $K_{X}$ the canonical divisor on $X$.

Given a non-singular projective surface $X$ and an ample divisor $H$ on $X$, the moduli space $M_{X,H}(n;c_{1},c_{2})$ of $H$-stable vector bundles of rank $n$ and Chern classes $c_{1}$ and $c_{2}$ was constructed by Maruyama \cite{Maruyama}. The Segre invariant for a rank two vector bundle $E$ on a surface is defined as follows
\begin{align*}
    S_{H}(E):=2\text{ min}\{\mu_{H}(E)-\mu_{H}(L)\},
\end{align*}
where $\mu_{H}(\_)$ is the $H$-slope (see Definition \ref{Definition 2.5mu}) and the minimum is taken over all line subbundles $L$ of $E$. This is a generalization of the Segre invariant for vector bundles on curves (see \cite{Roa}).

In \cite{Roa}, the purpose of the authors was to study vector bundles on $\mathbb{P}^{2}$, which constitutes a special geometric environment: its Picard group has rank one, and its homogeneous structure significantly restricts the variety of possible behaviors for stable bundles. From this perspective, the question remains open of how the Segre invariant behaves when the geometry of the base surface is substantially richer and admits a greater diversity of divisor classes and fiber structures. The aim of this work is to advance in that direction by studying the Segre invariant for the $H$-stable rank $2$ vector bundles on Hirzebruch surfaces.

As will be seen later, since the Segre invariant is invariant under tensor product with line bundles, the general discussion in the article can be reduced to rank two vector bundles whose first Chern class is normalized in $\text{Pic}(\Sigma_{n})$; that is $c_{1}\in\{0,\sigma,f,\sigma+f\}$.

Section $2$ introduces the Segre invariant on surfaces and collects some results that will be used. Section $3$ is the core of the paper. The main issue is to determine what numbers can appear as the Segre invariant an $H$-stable bundle on $\Sigma_{n}$ with fixed characteristic classes and how many global sections such a vector bundle has. The answer is given by the following result, where we say that a line subbundle is maximal in degree with respect to the ample divisor $H$, or equivalently as we will see later, it can be considered maximal with respect to the Segre invariant. 

\noindent\textbf{Theorem 3.4.}
Let $c_{1}\in\{0,\sigma,f,\sigma+f\}$ and $H$ be an ample divisor on $\Sigma_{n}$. Assume that one of the following conditions holds:
\begin{enumerate}
    \item $H\equiv \alpha\sigma+\beta f$, with $\alpha>1$, $\beta=\alpha n+1$ and $c_{2}>n>0$,

    \item $H\equiv \sigma+\beta f$ and $c_{2}\geq 2$,

    \item $H=H_{r}\equiv \alpha\sigma+\beta_{r}f$, with $\beta_{r}=\alpha n+r$, where $0<r\leq\alpha$ and $c_{2}>0$,

    \item $H\equiv \alpha\sigma+\beta f$, with $\alpha>0$, $\beta>\alpha n+(\alpha-1)$ and $c_{2}\geq 1$,

    \item $H=H_{r}\equiv \alpha\sigma+\beta_{r}f$, with $\beta_{r}=\alpha(n+1)+r$, where $0<r\leq\alpha$ and $c_{2}\geq 2$,

    \item $H=H_{r}\equiv \alpha\sigma+\beta_{r}f$, with $\beta_{r}=\alpha(n+1)+r$, where $0<r\leq\alpha$ and $c_{2}\geq 5$.
\end{enumerate}
Then there exists a vector bundle $E\in M_{\Sigma_{n},H}(2;c_{1},c_{2})$ with Segre invariant 
\[
S_H(E)=
\left\{
\begin{array}{c l}
2          & \text{if $c_{1}=0$ and $H$ is as in (1) or (2)},\\
r          & \text{if $c_{1}=\sigma$ and $H$ is as in (3) or $c_{1}=\sigma+f$ and $H$ is as in (5)},\\
\alpha     & \text{if $c_{1}=f$ and $H$ is as in (4)},\\
2\alpha+r  & \text{if $c_{1}=\sigma+f$ and $H$ is as in (6)}.
\end{array}
\right.
\]
Furthermore, $E$ fits in an exact sequence
\begin{align*}
    0\longrightarrow \mathcal{O}_{\Sigma_{n}}(D)\longrightarrow E\longrightarrow \mathcal{I}_{Z}(c_{1}-D)\longrightarrow 0,
\end{align*}
where
\[
D=
\begin{cases}
-\sigma, & \text{if $c_{1}=0$ and $H$ is as in (1)},\\
-f, & \text{if $c_{1}=0$ and $H$ is as in (2)},\\
0, & \text{if $c_{1}\in\{\sigma,f,\sigma+f\}$ and $H$ is as in (3,4,6)},\\
f, & \text{if $c_{1}=\sigma+f$ and $H$ is as in (5)}.
\end{cases}
\]
and $\mathcal{O}_{\Sigma_{n}}(D)\subset E$ is maximal and $Z$ is a general $0$-dimensional subscheme of length
\[
|Z|=
\begin{cases}
c_2-n, & \text{if } D=-\sigma, \\
c_2, & \text{if } D\in\{0,-f\},\\
c_2-1, & \text{if } D=f.
\end{cases}
\]

The proof strategy for the above theorem consists of applying Serre construction in order to obtain the desired extension on the Hirzebruch surface and then proving that the chosen line subbundle is maximal with respect to Mumford–Takemoto stability for a fixed polarization $H$.

In contrast to the case of $\mathbb{P}^2$ (see \cite{Roa}), the computation on $\Sigma_n$ reflects the structure of the Picard group and the dependence of the $H$-slope with respect to the chosen polarization.

It is worth emphasizing that, with the exception of case (1) and case (2) of Theorem \ref{Theorem 3.4}, all vector bundles under consideration admit at least one global section. This construction enables us to explicitly exhibit stable rank-two sheaves with global sections and, consequently, to demonstrate that some Brill-Noether loci for several polarizations and Chern classes are non-empty.

Although the main work focuses on the construction of vector bundles on Hirzebruch surfaces from line subbundles and $0$-dimensional subschemes that satisfy the Cayley–Bacharach property, there is previous work by Costa and Macías-Tarrío \cite{Costa3} that addresses similar constructions on ruled surfaces. Although both works employ the Cayley–Bacharach property and other related methods, the crucial difference lies in the techniques and purposes of constructing the vector bundles, their presentation as extensions, and the specific cases addressed. In particular, while the focus of our work is limited to Hirzebruch surfaces, the work of Costa and Macías-Tarrío, although more general, does not cover specific stability aspects. For example, their \cite[Theorem 4.8]{Costa3} and \cite[Proposition 4.12]{Costa3} only hold for a reduced family of ample divisors of the form $H\equiv \sigma+\beta f$. Another difference is that their extensions do not necessarily begin with a maximal line subbundle of the vector bundle, and do not provide structural
information regarding the varieties $W^{k}_{H}$ if $c_{1}\in \{\sigma, \sigma + f\}$. Therefore, our work provides a more precise characterization of geometric aspects of Hirzebruch surfaces, thus extending the analysis of the structure of vector bundles for this specific case of ruled surfaces.

The most important result in Brill-Noether theory that we obtained is Theorem \ref{Theorem 4.6}, which makes use of \cite[Corollary 2.6]{Filimon}, since for $E\in W^{k}_{H}\setminus W^{k+1}_{H}$, we obtain that $W^{k}$ is smooth and of the expected dimension $\rho^{k}_{H}$ at $E$ if the map $\mu_{E}$ is injective.

\noindent\textbf{Theorem 4.4.}
Let $E\in M_{\Sigma_{n},H}(2;c_{1},c_{2})$ with $c_{1}=c\sigma+df$ effective and $c_{2}>0$ such that $E$ lies in an extension of type
\begin{align*}
    0\longrightarrow \mathcal{O}_{\Sigma_{n}}(D)\longrightarrow E \longrightarrow \mathcal{I}_{Z}(c_{1}-D)\longrightarrow 0, \label{Example-muE}
\end{align*}
where $Z$ is a general $0$-dimensional subscheme of length $|Z|=c_{2}-D.(c_{1}-D)$ such that $|Z|=h^{0}(\mathcal{O}_{\Sigma_{n}}(c_{1}))$ where $D=a\sigma+bf$ is a base-point-free divisor on $\Sigma_{n}$ such that $h^{0}(\mathcal{O}_{\Sigma_{n}}(D))\geq 2$, $h^{1}(\mathcal{O}_{\Sigma_{n}}(2D))=0$, and $2(a+1)>c$ or $2(b+1)>d$. Then, if $h^{1}(\mathcal{O}_{\Sigma_{n}}(c_{1}))=0$ the map
\begin{align*}
    \mu_{E}:H^{0}(E)\otimes H^{1}(E^{*}\otimes\mathcal{O}_{\Sigma_{n}}(K_{\Sigma_{n}}))\longrightarrow H^{1}(E\otimes E^{*}\otimes\mathcal{O}_{\Sigma_{n}}(K_{\Sigma_{n}}))
\end{align*}
is injective.

The study of $W^{1}_{H}(2;c_{1},c_{2})\setminus W^{2}_{H}(2;c_{1},c_{2})$ in our cases by means of the map $\mu_{E}$ is done following the ideas of the \cite[Proposition 4.12]{Costa3} of Costa and Macías-Tarrío, where advantage is taken of the fact that the vector bundle has a unique section, and the injectivity of $\mu_{E}$ is shown to be equivalent to proving the injectivity of
\begin{align*}
    \tilde{\mu_{E}}:H^{1}(E^{*}\otimes\mathcal{O}_{X}(K_{X}))\longrightarrow H^{1}(E\otimes E^{*}\otimes\mathcal{O}_{X}(K_{X})),
\end{align*}
a technique that we generalize to our cases in Proposition \ref{Proposition 4.3.3}; however, Theorem \ref{Theorem 4.6} provides us with tools for the cases in which the vector bundle has at least two global sections. This is why this theorem is a direct tool for understanding Brill-Noether theory and extends previously known results for ruled surfaces.

\section{Preliminaries}
The goal of this section is to give the results concerning Hirzebruch surfaces and $H$-stability that we will use throughout this article.

\subsection{Hirzebruch surfaces}
\
\vspace{0.5em}

For any integer $n\geq 0$, let $\Sigma_{n}\cong\mathbb{P}(\mathcal{E})=\mathbb{P}(\mathcal{O}_{\mathbb{P}^{1}}\bigoplus\mathcal{O}_{\mathbb{P}^{1}}(-n))$ be a non-singular Hirzebruch surface . We denote by $\sigma$ and $f$ the standard basis of $\text{Pic}(\Sigma_{n})\cong\mathbb{Z}^{2}$ such that $\sigma^{2}=-n$, $\sigma.f=1$ and $f^{2}=0$. The canonical divisor is given by $K_{\Sigma_{n}}=-2\sigma-(n+2)f$.

\begin{remark}
It is well known that a divisor $D=a\sigma+bf$ on $\Sigma_{n}$ is effective if and only if $a,b\geq 0$ and $D$ is ample if and only if very ample, if and only if $a>0$ and $b>an$.
\end{remark}

Furthermore, the cohomology on Hirzebruch surfaces is completely determined. The following holds:

\begin{lemma}\cite[Lemma 2.2]{Costa2}\label{Lemma 1.23}
Let $\Sigma_{n}$ be a Hirzebruch surface. For any line bundle $\mathcal{O}_{\Sigma_{n}}(a\sigma+bf)$ on $\Sigma_{n}$ we have
\begin{align*}
H^{i}(\Sigma_{n},\mathcal{O}_{\Sigma_{n}}(a\sigma + bf)) = \left\{
\begin{array}{ll}
0 & \text{if }\hspace{1em} a = -1, \\ 
H^{i}(\mathbb{P}^{1}, S^{a}(\mathcal{F}) \otimes \mathcal{O}_{\mathbb{P}^{1}}(b)) & \text{if }\hspace{1em} a \geq 0, \\ 
H^{2-i}(\mathbb{P}^{1}, S^{-2-a}(\mathcal{F}) \otimes \mathcal{O}_{\mathbb{P}^{1}}(-n-b-2))^{*} & \text{if }\hspace{1em} a \leq -2
\end{array}
\right.
\end{align*}
where $S^{a}(\mathcal{F})$ denotes the $a$-th symmetric power of $\mathcal{F}=\mathcal{O}_{\mathbb{P}^{1}}\oplus\mathcal{O}_{\mathbb{P}^{1}}(-n)$.
\end{lemma}

\begin{proposition}\cite[Theorem 2.1(i)]{Coskun}\label{Proposition 2.3}
Let $L=\mathcal{O}_{\Sigma_{n}}(a\sigma+bf)$ be a line bundle on $\Sigma_{n}$. Then
\begin{align*}
    \chi(L)=(a+1)(b+1)-n{a(a+1)\over 2}
\end{align*}
\end{proposition}

\subsection{Sheaves, stability, and the Segre invariant}
\
\vspace{0.5em}

We begin this section by recalling some results on vector bundles on surfaces that we will use in the article. 

\begin{proposition}\label{Proposition 2.4}
Let $E$ be a torsion free sheaf of rank $r\geq 0$ on a non-singular projective surface $X$. Let $c_{1}$ and $c_{2}$ be the Chern classes of $E$. Then
\begin{align*}
    \chi(r;c_{1},c_{2}):=\chi(E)=r(1+p_{a}(X))-{c_{1}.K_{X}\over 2}+{c_{1}^{2}-2c_{2}\over 2},
\end{align*}
where $p_{a}(X)$ is the arithmetic genus.
\end{proposition}

\begin{definition}\label{Definition 2.5mu}
Let $X$ be a smooth, irreducible complex projective surface and let $H$ be an ample divisor on $X$. Let $\mathcal{E}$ be a torsion free coherent sheaf on $X$ with fixed Chern classes $c_{i}(\mathcal{E})\in H^{2i}(S,\mathbb{Z})$ for $i\in\{1,2\}$. The $H$-slope of $\mathcal{E}$ is defined as the rational number
\begin{align*}
    \mu_{H}(\mathcal{E}):={c_{1}(\mathcal{E}).H\over\hbox{rk}(\mathcal{E})},
\end{align*}
where $c_{1}(\mathcal{E}).H$ is the degree of $\mathcal{E}$ with respect to $H$ and is denoted by $\hbox{deg}_{H}(\mathcal{E})$. A torsion free coherent sheaf $\mathcal{E}$ is $H$-stable (respectively $H$-semistable) if for every nonzero subsheaf $\mathcal{F}$ of smaller rank, we have $\mu_{H}(\mathcal{F})<\mu_{H}(\mathcal{E})$ (respectively $\leq$).
\end{definition}

\begin{remark}\label{Lemma 4.2}
Let $\mathcal{E}$ be a torsion-free sheaf on a smooth projective variety $X$. Then $\mathcal{E}$ is stable if and only if for all line bundles $\mathcal{F}$ we have $\mathcal{E}\otimes\mathcal{F}$ is stable if and only if $\mathcal{E}^{*}$ is stable. For more details, see \cite[Lemma 4.5]{Friedman}.
\end{remark}

Let $\mathcal{O}_{X}(L)\hookrightarrow E$ be a line subbundle. By \cite[Proposition 2.5]{Friedman}, there exists a unique effective divisor $D \geq 0$ such that inclusion $\mathcal{O}_{X}(L)\hookrightarrow E$ factors through inclusion $\mathcal{O}_{X}(L)\hookrightarrow\mathcal{O}_{X}(L)\otimes\mathcal{O}_{X}(D)$, where the quotient $E/\mathcal{O}_{X}(L)\otimes\mathcal{O}_{X}(D)$ is torsion-free. Under this condition, if $D=0$, then there exists a local complete intersection codimension two subscheme $Z$ of $X$ and an exact sequence  
\[
0 \longrightarrow \mathcal{O}_{X}(L) \longrightarrow E \longrightarrow \mathcal{I}_Z(L') \longrightarrow 0,
\]
where $\mathcal{I}_{Z}(L')$ is torsion free and $\mathcal{I}_{Z}$ denotes the ideal sheaf of a subscheme $Z$ of codimension two. Note that $c_{1}(E)=c_{1}(L)+c_{1}(L')$ and $c_{2}(E)=c_{1}(L).c_{1}(L')+|Z|$, where $|Z|$ denotes the length of $Z$.

\begin{definition}\label{Definition 2.5}
Let $H$ be an ample divisor on $X$. For a vector bundle $E$ of rank $2$ on $X$ we define the \textit{Segre invariant} $S_{H}(E)$ by
\begin{align*}
    S_{H}(E):=2\hbox{ min}\{\mu_{H}(E)-\mu_{H}(L)\},
\end{align*}
where the minimum is taken over all line subbundles $L$ of $E$.
\end{definition}

Note that Definition \ref{Definition 2.5} is equivalent to $S_{H}(E)=\hbox{min }\{c_{1}(E).H-2L.H\}$. Furthermore, we have that as the minimum is taken over the line subbundles $L$ of $E$, then $S_{H}(E)=\hbox{deg}_{H}(E)-2\hbox{ max}\{\hbox{deg}_{H}(L)|L\subset E\}$.

For the Segre invariant to always be a finite number, the differences in slopes must reach their minimum; that is, the $\text{deg}_{H}(L)$ of the line subbundles of the vector bundle must be bounded above with respect to the ample divisor:

\begin{lemma}\cite[Lemma 2.5]{Roa}
Let $X$ be a projective smooth surface, $H$ an ample line bundle on $X$ and $E$ a rank $2$ vector bundle on $X$. Then the set $\{L.H\hspace{0.5em}|\hspace{0.5em}L\subset E\text{, }L\text{ a line bundle}\}$ is bounded from above.
\end{lemma}

The term invariant is used because $S_{H}(E)=S_{H}(E\otimes L)$ for any line bundle $L\in\text{Pic}(X)$. Furthermore, a vector bundle $E$ is $H$-stable (resp. $H$-semistable) if and only if $S_{H}(E)>0$ (resp. $\geq 0$), see \cite{Roa} for details.

\begin{definition}
A line subbundle $\mathcal{O}_{X}(L)\subset E$ is maximal if 
\begin{align*}
    S_{H}(E) = \hbox{deg}_{H}(E)-2\hbox{deg}_{H}(\mathcal{O}_{X}(L)).
\end{align*}
\end{definition}

\subsection{Brill-Noether Theory}
\
\vspace{0.5em}

Our aim is to study the moduli space of stable vector bundles Hirzebruch surfaces using the Brill-Noether locus. For this reason, we conclude this section by recalling some basic facts and presenting results from Costa and Macías-Tarrío \cite{Costa3}, in order to make a subsequent comparison between our results.

\begin{proposition}\label{Proposition 4.4}
Let $X$ be a smooth, projective rational surface with effective anticanonical
line bundle and let $H$ be an ample line bundle on $X$. Then, the moduli space $M_{X,H}(2;c_{1},c_{2})$ of rank two, $H$-stable vector bundles $E$ on $X$ with fixed Chern classes $c_{i}(E)=c_{i}$ is either empty or a smooth irreducible variety of dimension $$\text{dim}(M_{X,H}(2;c_{1},c_{2}))=4c_{2}-c_{1}^{2}-3.$$
\end{proposition}

\begin{corollary}\cite[Corollary 2.8]{Costa4}\label{Corollary 4.5}
Let $X$ be a smooth projective surface and let $M_{H}=M_{X,H}(r;c_{1},c_{2})$ be a
moduli space of rank $r$, $H$-stable vector bundles $E$ on $X$ with fixed Chern classes $c_{i}(E)=c_{i}$. Assume that $c_{1}.H\geq rK_{X}.H$. Then, for any $k\geq 0$, there exists a determinantal variety $W^{k}_{H}(r;c_{1},c_{2})$ such that
\begin{align*}
    \text{Supp}(W^{k}_{H}(r;c_{1},c_{2}))=\{E\in M_{H}|h^{0}(E)\geq k\}.
\end{align*}
Moreover, each non-empty irreducible component of $W^{k}_{H}(r;c_{1},c_{2})$ has dimension greater than or equal to 
\begin{align*}
    \rho^{k}_{H}(r;c_{1},c_{2})=\text{dim}(M_{H})-k(k- r(1+P_{a}(X))+{c_{1}.K_{X}\over2}-{c_{1}^{2}\over2}+c_{2})
\end{align*}
and $W^{k+1}_{H}(r;c_{1},c_{2})\subset\text{Sing}(W^{k}_{H}(r;c_{1},c_{2}))$ whenever $W^{k}_{H}(r;c_{1},c_{2})\neq M_{X,H}(r;c_{1},c_{2})$.
\end{corollary}

In \cite{Costa3}, one of the main problems that Costa and Macías-Tarrío focus on is determining when the Brill-Noether locus $W^{k}_{H}$ is non-empty. One of their main theorems is the following:

\begin{theorem}\cite[Theorem 4.1]{Costa3}\label{Theorem 1.2 Costa}
Let $X$ be a ruled surface over a nonsingular curve $C$ of genus $g\geq 0$, $m\in\{0,1\}$, $c_{2}>>0$ an integer and $H\equiv \alpha C_{0}+\beta f$ an ample divisor on $X$ with $\alpha(n+m)<\beta$. Then, for any $k$ in the range 
\begin{align*}
    \text{max}\{1,g\}\leq k<{1\over2\alpha}[\beta-\alpha(e-m+2g-2)],
\end{align*}
the Brill-Noether locus $W^{k}_{H}(2;C_{0}+\mathfrak{m}f,c_{2})\neq\varnothing$ with $m=\text{deg}(\mathfrak{m})$ whenever $M_{H}(2;C_{0}+\mathfrak{m}f,c_{2})\neq\varnothing$.
\end{theorem}

\begin{corollary}\cite[Corollary 4.11]{Costa3}\label{Corollary 3.10 Costa}
Let $X$ be a ruled surface over a nonsingular curve $C$ of genus $g\geq 0$, $c_{2}>>0$ an integer and $H\equiv\sigma+\beta f$ an ample divisor on $X$. Then, $W^{k}_{H}(2;f,c_{2})\neq\varnothing$ if and only if $1\leq k\leq 2$.
\end{corollary}

Corollary \ref{Corollary 3.10 Costa} ensures that $W^{k}_{H}(2;f,c_{2})\neq\varnothing$ on $\Sigma_{n}$ if $k=1,2$ and $c_{2}>>0$, but only for the ample divisor $H\equiv\sigma+\beta f$. In our case, we prove the same, but for a general ample divisor $H\equiv\alpha\sigma+\beta f$, but much more restrictive for the case $k=2$ with respect to $c_{2}$ as will be seen in case (4) of Theorem \ref{Theorem 3.4}.

L. Filimon in \cite{Filimon} gives a smoothness criterion for the varieties $W^{k}_{H}$. For $E\in W^{k}_{H}\setminus W^{k+1}_{H}$ we will consider the map
\begin{align*}
    \mu_{E}:H^{0}(E)\otimes H^{1}(E^{*}\otimes\mathcal{O}_{X}(K_{X}))\longrightarrow H^{1}(E\otimes E^{*}\otimes\mathcal{O}_{X}(K_{X})).
\end{align*}
Then the smoothness of $W^{k}_{H}$ will be linked to the injectivity of $\mu_{E}$. More precisely, we have the following:

\begin{corollary}\cite[Corollary 2.6]{Filimon} \label{Corollary 4.1}
Let $E\in W_{H}^{k}\setminus W_{H}^{k+1}$. Then the Zariski tangent space of $W_{H}^{k}$ at $E$ is
\begin{align*}
    T_{E}W_{H}^{k}\cong(\text{Im}(\mu_{E}))^{\perp}.
\end{align*}
Moreover, if $M_{H}$ is smooth at $E$, then $W_{H}^{k}$ is smooth and of expected dimension $\rho_{H}^{k}$ at $E$ if and only if $\mu_{E}$ is injective.
\end{corollary}

\section{Segre invariant for rank 2 vector bundles on $\Sigma_{n}$}

In this section we study the Segre invariant for vector bundles of rank $2$ on $\Sigma_{n}$. The Segre invariant depends on the choice of an ample divisor $H$ on the surface $\Sigma_{n}$. Since $\text{Pic}(\mathbb{P}^2)\cong\mathbb{Z}$, there exists a unique notion of $H$-stability on $\mathbb{P}^2$, determined by the ample generator $\mathcal{O}_{\mathbb{P}^2}(1)$. Accordingly, in \cite{Roa}, the authors work with a unique stability condition. In contrast, Hirzebruch surfaces admit infinitely many ample divisors, and hence infinitely many notions of $H$-stability.

Since a rank two vector bundle $E$ on $\Sigma_{n}$ is $H$-stable if, and only if, $E\otimes\mathcal{O}_{\Sigma_{n}}(D)$ is $H$-stable for any divisor $D\in\text{Pic}(\Sigma_{n})$, we may assume, without loss of generality, that
$c_{1}(E)$ is one of the following: $0$, $\sigma$, $f$ or $\sigma+f$.

The technique we follow for the construction of the rank two vector bundles on the Hirzebruch surfaces is to construct the given bundle by a short exact sequence using the Cayley-Bacharach property \cite[Part II-Section 5.1]{Huybrechts}, then prove the $H$-stability, and finally prove that the line subbundle at the beginning of the sequence is a maximal line subbundle of the vector bundle.

\begin{theorem}\cite[Theorem 5.1.1]{Huybrechts}
Let $Z\subset X$ be a local complete intersection of codimension two, and let $L$ and $M$ be line bundles on $X$. Then there exists a nontrivial extension of the form
\begin{align*}
    0 \longrightarrow L\longrightarrow E\longrightarrow\mathcal{I}_{Z}(M) \longrightarrow 0
\end{align*}
such that $E$ is locally free if and only if the pair $(L^{*}\otimes M\otimes K_{X}, Z)$ has the Cayley-Bacharach property:

{\bf (CB):} If $Z'\subset Z$ is a subscheme with $|Z'|=|Z|-1$ and $s\in H^{0}(S, L^{*}\otimes M\otimes K_{X})$ with $s|_{Z'}=0$, then $s|_{Z}=0$.
\end{theorem}

Note that the CB property clearly holds for all $Z$ if $H^{0}(X,L^{*}\otimes M\otimes K_{X})=0$.

\begin{proposition}\cite[Proposition 2.4]{Costa3}\label{Proposition 1.27}
Let $X$ be a smooth projective surface, $D$ a divisor on $X$ and $Z\subset X$ a general $0$-dimensional subscheme. If $|Z|\geq h^{0}(\mathcal{O}_{X}(D))$ then $h^{0}(\mathcal{I}_{Z}(D))=0$.
\end{proposition}

\begin{remark}\label{Observation 2.18}
Let $L\in\hbox{Pic}(X)$ non-effective and $\mathcal{I}_{Z}$ be an ideal sheaf with $Z\subset X$ a general $0$-dimensional subescheme. Then $\mathcal{O}_{X}(L)\hookrightarrow\mathcal{I}_{Z}$ exists if and only if there is an effective divisor $D$ such that $Z\subset D$, with $D$ linearly equivalent to $-L$.
\end{remark}

Let $L=a\sigma+bf$, $M=c\sigma+df$ in $\hbox{Pic}(\Sigma_{n})$ and $Z\subset\Sigma_{n}$ be a  local complete intersection of codimension two. For our interest, since we only work with four cases for our porpuses $c_{1}$ and we want there to exist a vector bundle of rank two $E$ on $\Sigma_{n}$ such that
\begin{align*}
    0\longrightarrow\mathcal{O}_{\Sigma_{n}}(L)\longrightarrow E\longrightarrow\mathcal{I}_{Z}(M)\longrightarrow 0,
\end{align*}
by \cite[Proposition 2.5]{Friedman} it is necessary that $a+c=\delta_{1}$ and $b+d=\delta_{2}$ where $\delta_{i}\in\{0,1\}$ and $c_{1}=\delta_{1}\sigma+\delta_{2}f$.  On the other hand, since $L^{*}\otimes M\otimes K_{\Sigma_{n}}=(c-a-2)\sigma+(d-b-n-2)f$, we have that $H^{0}(\Sigma_{n};\mathcal{O}_{\Sigma_{n}}(L^{*}\otimes M\otimes K_{\Sigma_{n}}))=0$ if $c-a-2<0$ or $d-b-n-2<0$. With this show that if $a\geq 0$ or $b\geq 0$, then the CB property is trivially satisfied.

\begin{remark}
The particular ample divisors $H$ featured in the following theorem (some parameterized by a scalar $t$ or $r$) arise naturally from the analysis of the stability conditions. The restriction $0<t,r\leq\alpha$ in the cases (3), (5) and (6) of Theorem \ref{Theorem 3.4} precisely delineates the region within the ample cone where the extensions under consideration preserve $H$-stability, thereby establishing the viability of the construction.
\end{remark}

\begin{theorem}\label{Theorem 3.4}
Let $c_{1}\in\{0,\sigma,f,\sigma+f\}$ and $H$ be an ample divisor on $\Sigma_{n}$. Assume that one of the following conditions holds:
\begin{enumerate}
    \item $H\equiv \alpha\sigma+\beta f$, with $\alpha>1$, $\beta=\alpha n+1$ and $c_{2}>n>0$,

    \item $H\equiv \sigma+\beta f$ and $c_{2}\geq 2$,

    \item $H=H_{t}\equiv \alpha\sigma+\beta_{t}f$, with $\beta_{t}=\alpha n+t$, where $0<t\leq\alpha$ and $c_{2}>0$,

    \item $H\equiv \alpha\sigma+\beta f$, with $\alpha>0$, $\beta>\alpha n+(\alpha-1)$ and $c_{2}\geq 1$, 

    \item $H=H_{r}\equiv \alpha\sigma+\beta_{r}f$, with $\beta_{r}=\alpha(n+1)+r$, where $0<r\leq\alpha$ and $c_{2}\geq 2$,

    \item $H=H_{r}\equiv \alpha\sigma+\beta_{r}f$, with $\beta_{r}=\alpha(n+1)+r$, where $0<r\leq\alpha$ and $c_{2}\geq 5$.
\end{enumerate}
Then there exists a vector bundle $E\in M_{\Sigma_{n},H}(2;c_{1},c_{2})$ with Segre invariant 
\[
S_H(E)=
\left\{
\begin{array}{c l}
2          & \text{if $c_{1}=0$ and $H$ is as in (1) or (2)},\\
r          & \text{if $c_{1}=\sigma$ and $H$ is as in (3) or $c_{1}=\sigma+f$                     and $H$ is as in (5)},\\
\alpha     & \text{if $c_{1}=f$ and $H$ is as in (4)},\\
2\alpha+r  & \text{if $c_{1}=\sigma+f$ and $H$ is as in (6)}.
\end{array}
\right.
\]
Furthermore, $E$ fits in an exact sequence
\begin{align}
    0\longrightarrow \mathcal{O}_{\Sigma_{n}}(L)\longrightarrow E\longrightarrow \mathcal{I}_{Z}(c_{1}-L)\longrightarrow 0,
\end{align}
where
\[
L=
\begin{cases}
-\sigma, & \text{if $c_{1}=0$ and $H$ is as in (1)},\\
-f, & \text{if $c_{1}=0$ and $H$ is as in (2)},\\
0, & \text{if $c_{1}\in\{\sigma,f,\sigma+f\}$ and $H$ is as in (3,4,6)},\\
f, & \text{if $c_{1}=\sigma+f$ and $H$ is as in (5)}.
\end{cases}
\]
and $\mathcal{O}_{\Sigma_{n}}(L)\subset E$ is maximal and $Z$ is a general $0$-dimensional subscheme of length
\[
|Z|=
\begin{cases}
c_2-n, & \text{if } L=-\sigma, \\
c_2, & \text{if } L\in\{0,-f\},\\
c_2-1, & \text{if } L=f.
\end{cases}
\]
\end{theorem}

\begin{proof}
Note that if $L\in\{-\sigma,-f,0,f\}$ then we have that $H^{0}(\mathcal{O}_{\Sigma_{n}}(L^{*}\otimes M\otimes K_{\Sigma_{n}}))=0$ where $M=c_{1}-L$, and therefore the CB property is trivially satisfied for any $0$-dimensional subscheme $Z$ in each of the following cases. Furthermore, note that if $E$ is a vector bundle of rank two on a Hirzebruch surface $\Sigma_{n}$ given by an extension of type
\begin{align*}
        0 \longrightarrow \mathcal{O}_{\Sigma_{n}}(L)\longrightarrow E\longrightarrow \mathcal{I}_{Z}(c_{1}-L) \longrightarrow 0, 
    \end{align*}
then, if $\mathcal{O}_{\Sigma_{n}}(J)\hookrightarrow E$ is a line subbundle, we have two possibilities: $\mathcal{O}_{\Sigma_{n}}(J)\hookrightarrow\mathcal{O}_{\Sigma_{n}}(L)$ or $\mathcal{O}_{\Sigma_{n}}(J)\hookrightarrow\mathcal{I}_{Z}(c_{1}-L)$. If $\mathcal{O}_{\Sigma_{n}}(J)\hookrightarrow\mathcal{O}_{\Sigma_{n}}(L)$, then $L-J$ is effective, so if $H$ is an ample divisor on $\Sigma_{n}$, then $(L-J).H\geq 0$, that is, $\mu_{H}(\mathcal{O}_{\Sigma_{n}}(J))\leq\mu_{H}(\mathcal{O}_{\Sigma_{n}}(L))$. Therefore, if we want to verify that a line subbundle $\mathcal{O}_{\Sigma_{n}}(J)$ such that $\mathcal{O}_{\Sigma_{n}}(J)\hookrightarrow E$ and $\mathcal{O}_{\Sigma_{n}}(J)\hookrightarrow\mathcal{O}_{\Sigma_{n}}(L)$ does not destabilize $E$, it suffices to show that $\mu_{H}(\mathcal{O}_{\Sigma_{n}}(L))\leq\mu_{H}(E)$. In the following proofs, it can be easily verified that the condition $\mu_{H}(\mathcal{O}_{\Sigma_{n}}(L))\leq\mu_{H}(E)$ is always met, therefore for the stability of the vector bundle $E$ we will focus on when the line subbundle $\mathcal{O}_{\Sigma_{n}}(J)\hookrightarrow E$ is such that $\mathcal{O}_{\Sigma_{n}}(J)\hookrightarrow\mathcal{I}_{Z}(c_{1}-L)$.

\begin{enumerate}[leftmargin=*]
    \item Assume that $H\equiv \alpha\sigma+\beta f$, with $\alpha>1$, $\beta=\alpha n+1$ and $c_{2}>n>0$. Let $Z\subset\Sigma_{n}$ be a general $0$-dimensional subscheme, therefore $Z\not\subset\sigma$, and let $L=-\sigma$ and $M=\sigma$ be line bundles on $\Sigma_{n}$. Then by CB property there exists a non trivial extension of the form
    \begin{align}
        0 \longrightarrow \mathcal{O}_{\Sigma_{n}}(-\sigma)\longrightarrow E\longrightarrow \mathcal{I}_{Z}(\sigma) \longrightarrow 0,  \label{Teorema 3.1suc}
    \end{align}
    where $E$ is locally free  and has Chern classes $c_{1}=0$ and $c_{2}=|Z|-n$.

    Let us prove that the vector bundle $E$ is $H$-stable. Assume $\mathcal{O}_{\Sigma_{n}}(J)\hookrightarrow\mathcal{I}_{Z}(\sigma)$ where $J=c\sigma+df$. Then $\sigma-J=(1-c)\sigma-df$ is an effective divisor, and therefore $c\leq 1$ and $d\leq 0$.

    Suppose that $c=1$, so $J=\sigma+df$. If $d=0$, then $\mathcal{O}_{\Sigma_{n}}(\sigma)\hookrightarrow\mathcal{I}_{Z}(\sigma)$; that is, $\mathcal{O}_{\Sigma_{n}}\hookrightarrow\mathcal{I}_{Z}$, but this is possible if and only if $Z=\varnothing$ which is a contradiction, and therefore $d<0$. We have
    \begin{align*}
        \mu_{H}(\mathcal{O}_{\Sigma_{n}}(J)) = \beta-\alpha n+d\alpha = 1+d\alpha < 0 < \mu_{H}(E).
    \end{align*}
    Now, suppose that $c\leq 0$, and in this case, $\mu_{H}(\mathcal{O}_{\Sigma_{n}}(J))= c+d\alpha$.

    If $c<0$, since $d\alpha\leq 0$, then $\mu_{H}(\mathcal{O}_{\Sigma_{n}}(J))<0=\mu_{H}(E)$. Then assume that $c=0$. If $d=0$, then $J=0$ and $\mathcal{O}_{\Sigma_{n}}\hookrightarrow\mathcal{I}_{Z}(\sigma)$; that is, $\mathcal{O}_{\Sigma_{n}}(-\sigma)\hookrightarrow\mathcal{I}_{Z}$, but by Remark \ref{Observation 2.18} this is only possible if and only if there exists $D\in|\sigma|$ with $Z\subset D$, but $\text{dim}\hspace{0.1em}|\sigma|=h^{0}(\mathcal{O}_{\Sigma_{n}}(\sigma))-1=0$ for $n>0$; then $D=\sigma$, but $Z\not\subset\sigma$, and therefore $d\neq 0$. Thus, $d<0$ and in this case,
    \begin{align*}
        \mu_{H}(\mathcal{O}_{\Sigma_{n}}(J))=d\alpha<0=\mu_{H}(E).
    \end{align*}

    Then $E$ is $H$-stable, and therefore $S_{H}(E)>0$. Since $c_{1}=0$, then $\text{deg}_{H}(E)=0$ and $S_{H}(E)$ is even, furthermore, it follows that
    \begin{align*}
        \text{deg}_{H}(E)-2\text{deg}_{H}(\mathcal{O}_{\Sigma_{n}}(L)) = -2\text{deg}_{H}(\mathcal{O}_{\Sigma_{n}}(L)) = -2(-\sigma).(\alpha\sigma+(\alpha n+1)f) = 2,
    \end{align*}
    which is the smallest positive even integer, that is, $S_{H}(E)=2$, and therefore, $\mathcal{O}_{\Sigma_{n}}(-\sigma)$ is a maximal line subbundle of $E$.

    Let us compute the global sections of $E$. Since $h^{i}(\mathcal{O}_{\Sigma_{n}}(-\sigma))=0$ for $i\in\{0,1\}$ then, from the exact sequence in cohomology induced by \eqref{Teorema 3.1suc} we have that $h^{0}(E)=h^{0}(\mathcal{I}_{Z}(\sigma))$, but $|Z|=c_{2}-n\geq 1$ with $Z$ general, and for $n>0$ we have $h^{0}(\mathcal{O}_{\Sigma_{n}}(\sigma))=1$, then by Proposition \ref{Proposition 1.27}, we have $h^{0}(\mathcal{I}_{Z}(\sigma))=0$; that is, $E$ does not have global sections.

    \item Assume that $H\equiv \sigma+\beta f$ and $c_{2}\geq 2$. Let $Z\subset\Sigma_{n}$ be a general $0$-dimensional subscheme, and $L=-f$ and $M=f$ be line bundles on $\Sigma_{n}$. Then by CB property there exists a non-trivial extension of the form 
    \begin{align}
        0 \longrightarrow \mathcal{O}_{\Sigma_{n}}(-f)\longrightarrow E\longrightarrow \mathcal{I}_{Z}(f)\longrightarrow 0, \label{Teorema 3.2.suc}
    \end{align}
    where $E$ is locally free and has Chern classes $c_{1}=0$ and $c_{2}=|Z|$.
    
    The proof that the vector bundle $E$ is $H$-stable, has Segre invariant $S_{H}(E)=2$, and has no global sections is analogous to that of case (1), therefore, we omit the details.

    \item Assume that $H=H_{t}\equiv \alpha\sigma+\beta_{t}f$, with $\beta_{t}=\alpha n+t$, where $0<t\leq\alpha$ and $c_{2}>0$. Let $Z\subset\Sigma_{n}$ be a general $0$-dimensional subscheme, and $L=0$ and $M=\sigma$ be line bundles on $\Sigma_{n}$. Then by CB property there exists a non-trivial extension of the form 
    \begin{align}
        0 \longrightarrow \mathcal{O}_{\Sigma_{n}}\longrightarrow E\longrightarrow \mathcal{I}_{Z}(\sigma) \longrightarrow 0, \label{Teorema 3.3.suc}
    \end{align}
    where $E$ is locally free and has Chern classes $c_{1}=\sigma$ and $c_{2}=|Z|$.

    Now, let us prove that the vector bundle $E$ is $H$-stable. Assume $\mathcal{O}_{\Sigma_{n}}(J)\hookrightarrow\mathcal{I}_{Z}(\sigma)$, where $J=c\sigma+df$. In this case $\sigma-J=(1-c)\sigma-df$ is an effective divisor, and therefore $c\leq 1$ and $d\leq 0$. Suppose that $c=1$ and $d=0$; that is $J=\sigma$  and since $h^{0}(\mathcal{O}_{\Sigma_{n}}(-\sigma))=0$ and $h^{0}(\mathcal{I}_{Z})=0$, if we tensor \eqref{Teorema 3.3.suc} by $\mathcal{O}_{\Sigma_{n}}(-\sigma)$, we get $1\leq h^{0}(E(-\sigma))\leq 0$, which is impossible, and therefore, $J\neq\sigma$. From the above, if $c=1$, then $d<0$, and if $c<1$, then $d\leq 0$.

    Assume that $c=1$. Then
    \begin{align*}
        \mu_{H}(\mathcal{O}_{\Sigma_{n}}(J)) = c(\beta-\alpha n)+d \alpha = t+d\alpha \leq 0 < {t\over 2} = \mu_{H}(E).
    \end{align*}

    Assume that $c<1$. Then 
    \begin{align*}
        \mu_{H}(\mathcal{O}_{\Sigma_{n}}(J)) = c(\beta-\alpha n)+d \alpha \leq c(\beta-\alpha n) = ct \leq 0 < {t\over 2} = \mu_{H}(E).
    \end{align*}
    Therefore, in any case, $\mu_{H}(\mathcal{O}_{\Sigma_{n}}(J))<\mu_{H}(E)$. Then $E$ is $H$-stable, and therefore $S_{H}(E)>0$. Suppose $\mathcal{O}_{\Sigma_{n}}(G)\hookrightarrow E$ is a maximal line subbundle of $E$. Then
    \begin{align*}
        0\leq\hbox{deg}_{H}(\mathcal{O}_{\Sigma_{n}}(G))<{t\over 2}
    \end{align*}
    Note that $\hbox{deg}_{H}(E)-2\hbox{deg}_{H}(\mathcal{O}_{\Sigma_{n}}(L))=t$, and if $S_{H}(E)<t$, then $\hbox{deg}_{H}(\mathcal{O}_{\Sigma_{n}}(G))>0$. Suppose that, indeed, $0<\hbox{deg}_{H}(\mathcal{O}_{\Sigma_{n}}(G))<{t\over 2}$. Tensoring the exact sequence \eqref{Teorema 3.3.suc} by $\mathcal{O}_{\Sigma_{n}}(-G)$ and we obtain:
    \begin{align}
        0 \longrightarrow \mathcal{O}_{\Sigma_{n}}(-G)\longrightarrow E(-G)\longrightarrow \mathcal{I}_{Z}(\sigma-G) \longrightarrow 0, \label{Example 6}
    \end{align}
    If $\sigma-G$ is an effective divisor with $G=a\sigma+bf$ then $a\leq 1$ and $b\leq 0$.

    Assume that $a=1$. Since $\sigma-G$ is effective, then $(\sigma-G).H = (-bf).(\alpha\sigma+\beta_{t} f) = -\alpha b > 0$, therefore $b<0$, since $G\neq L$. Furthermore, in this case we have that $\text{deg}_{H}(\mathcal{O}_{\Sigma_{n}}(G)) = (\beta-\alpha n)+b\alpha = t+b\alpha \leq 0$, which is impossible, and therefore, $a\neq 1$; that is, $a\leq 0$.

    With the above, we have that $0 < \hbox{deg}_{H}(\mathcal{O}_{\Sigma_{n}}(G)) = at+b\alpha < 0$, which is impossible, and therefore, $\sigma-G=(1-a)\sigma-bf$ cannot be effective, and then $a>1$ or $b>0$.

    If we assume that $a>1$ and $b>0$, then
    \begin{align*}
        \mu_{H}(\mathcal{O}_{\Sigma_{n}}(G)) = at+b(\alpha) > t > {t\over 2} = \mu_{H}(E),
    \end{align*}
    which contradicts the $H$-stability of $E$. So either $a>1$ or $b>0$, but not both.

    Recall that, since $\mathcal{O}_{\Sigma_{n}}(G)$ is a maximal line subbundle of $E$, then there also exists an exact sequence of the form
    \begin{align*}
        0 \longrightarrow \mathcal{O}_{\Sigma_{n}}(G)\longrightarrow E\longrightarrow \mathcal{I}_{Z'}(\sigma-G) \longrightarrow 0.
    \end{align*}
    by \cite[Proposition 2.5]{Friedman} where $Z'\subset\Sigma_{n}$ is a $0$-dimensional subscheme of length $|Z'|=c_{2}-G.(\sigma-G)$. Based on the above and according to the values of $a$ and $b$, the following cases are obtained:

    {\bf Case (a)} Assume that $a>1$, whence $b\leq 0$. If $b=0$, then $\text{deg}_{H}(\mathcal{O}_{\Sigma_{n}}(G))=at>{t\over 2}$, which is impossible. Then $b<0$ and therefore $G=a\sigma+bf$ is non-effective. Thus, since $G$ and $\sigma-G$ are not effective, then $h^{0}(E)=0$.

    {\bf Case (b)} Assume that $b>0$, whence $a\leq 1$. If $a\in\{0,1\}$ we have
    \begin{align*}
        \mu_{H}(\mathcal{O}_{\Sigma_{n}}(G)) = at+b\alpha \geq b\alpha \geq t > {t\over 2} = \mu_{H}(E),
    \end{align*}
    which is impossible. Then $a<0$ and therefore $G=a\sigma+bf$ is non-effective. Thus, since $G$ and $\sigma-G$ are not effective, then $h^{0}(E)=0$.

    Then it follows that if $E$ is a rank $2$ vector bundle on $\Sigma_{n}$ with $c_{1}(E)=\sigma$ such that a maximal line subbundle $\mathcal{O}_{\Sigma_{n}}(G)$ is such that $0<\text{deg}_{H}(\mathcal{O}_{\Sigma_{n}}(G))<{t\over 2}$ then $h^{0}(E)=0$. But from the original sequence \eqref{Teorema 3.3.suc} we have that $h^{0}(E)\geq 1$, and so $\mathcal{O}_{\Sigma_{n}}(G)$ cannot exist. Then $\mathcal{O}_{\Sigma_{n}}(L)=\mathcal{O}_{\Sigma_{n}}$ is a maximal line subbundle of $E$ and $S_{H}(E)=t$.

    Let us compute the global sections of $E$. Since $h^{0}(\mathcal{O}_{\Sigma_{n}})=1$ and $h^{1}(\mathcal{O}_{\Sigma_{n}})=0$ then, from the exact sequence in cohomology induced by \eqref{Teorema 3.3.suc} we have that $h^{0}(E)=1+h^{0}(\mathcal{I}_{Z}(\sigma))$.

    Suppose that $n\geq 1$. Since $|Z|=c_{2}\geq 1=h^{0}(\mathcal{O}_{\Sigma_{n}}(\sigma))$ with $Z$ general, then $h^{0}(\mathcal{I}_{Z}(\sigma))=0$ by Proposition \ref{Proposition 1.27}, so $h^{0}(E)=1$ in this case.

    Now, suppose that $n=0$. If $|Z|=c_{2}\geq 2=h^{0}(\mathcal{O}_{\Sigma_{0}}(\sigma))$ with $Z$ general, then $h^{0}(\mathcal{I}_{Z}(\sigma))=0$, so $h^{0}(E)=1$. Assume that $c_{2}=1$; that is, $Z=\{p\}$, and note that in the exact sequence in cohomology induced by
    \begin{align*}
        0 \longrightarrow \mathcal{I}_{Z}(\sigma)\longrightarrow \mathcal{O}_{\Sigma_{0}}(\sigma)\longrightarrow \mathcal{O}_{Z}(\sigma)\longrightarrow 0,
    \end{align*}
    all cohomology groups starting from $H^{1}(\mathcal{O}_{\Sigma_{0}}(\sigma))$ are zero, $h^{0}(\mathcal{O}_{\Sigma_{0}}(\sigma))=2$ and $h^{0}(\mathcal{O}_{Z}(\sigma))=1$ by Lemma \ref{Lemma 1.23}. Now, since $Z=\{p\}$, the morphism $H^{0}(\mathcal{O}_{\Sigma_{0}}(\sigma)) \rightarrow H^{0}(\mathcal{O}_{Z}(\sigma))$ is just the evaluation morphism at the point $p$, let us call it $\mathrm{ev}_{p}$; that is
    \[
    \begin{aligned}
    \mathrm{ev}_p \colon\; 
    H^{0}\big(\mathcal{O}_{\Sigma_{0}}(\sigma)\big) &\longrightarrow H^{0}\big(\mathcal{O}_{Z}(\sigma)\big) \cong \mathbb{C} \\
    s &\longmapsto s(p).
    \end{aligned}
    \]
    As the linear system $|\sigma|$ has no base points if $n=0$, then for any point of $\Sigma_{0}$ there exists a global section of $\mathcal{O}_{\Sigma_{0}}(\sigma)$ which comes from an element of $|\sigma|$ that does not vanish at $p$. Hence, $\mathrm{ev}_p$ is surjective, yielding $\text{dim}(\text{Ker}(\mathrm{ev}_p))=1$. Therefore, $h^{0}(\mathcal{I}_{Z}(\sigma))=1$, and then $h^{0}(E)=2$ in this case.

    Finally, we can summarize everything as follows:
    \[
    h^{0}(E) = 
    \begin{cases}
    2 & \text{if } n=0 \text{ and } c_{2}=1, \\
    1 & \text{if } n=0 \text{ and } c_{2}\geq 2, \\
    1 & \text{if } n\geq 1. \\
    \end{cases}
    \]

    \item First assume that $H\equiv \alpha\sigma+\beta f$, with $\alpha>1$, $\beta>\alpha n+(\alpha-1)$ and $c_{2}\geq 1$. Let $Z\subset\Sigma_{n}$ be a general $0$-dimensional subscheme, and let $L=0$ and $M=f$ be line bundles on $\Sigma_{n}$. Then by CB property there exists a non-trivial extension of the form
    \begin{align}
        0 \longrightarrow \mathcal{O}_{\Sigma_{n}}\longrightarrow E\longrightarrow \mathcal{I}_{Z}(f) \longrightarrow 0, \label{Teorema 3.4suc}
    \end{align}
    where $E$ is locally free and has Chern classes $c_{1}=f$ and $c_{2}=|Z|$.

    Now, let us prove that the vector bundle $E$ is $H$-stable. Assume $\mathcal{O}_{\Sigma_{n}}(J)\hookrightarrow\mathcal{I}_{Z}(f)$, where $J=c\sigma+df$. In this case $ f-J=-c\sigma+(1-d)f$ is an effective divisor, and therefore $c\leq 0$ and $d\leq 1$. Suppose that $c=0$ and $d=1$; that is $J=f$ and since $h^{0}(\mathcal{O}_{\Sigma_{n}}(-f))=0$ and $h^{0}(\mathcal{I}_{Z})=0$, if we tensor \eqref{Teorema 3.4suc} by $\mathcal{O}_{\Sigma_{n}}(-J)$, we get $1\leq h^{0}(E(-f))\leq 0$, which is impossible, and therefore, $J\neq f$. From the above, if $c=0$, then $d<1$, and if $c<0$, then $d\leq 1$.

    Assume $c=0$. Then
    \begin{align*}
        \mu_{H}(\mathcal{O}_{\Sigma_{n}}(J)) = c(\beta-\alpha n)+\alpha d = \alpha d \leq 0 < {\alpha\over 2} = \mu_{H}(E).
    \end{align*}
    Assume $c<0$. Then
    \begin{align*}
        \mu_{H}(\mathcal{O}_{\Sigma_{n}}(J)) = c(\beta-\alpha n)+\alpha d \leq -(\beta-\alpha n)+\alpha < (1-\alpha)+\alpha = 1 \leq {\alpha\over 2} = \mu_{H}(E).
    \end{align*}
    Then $E$ is $H$-stable, and therefore $S_{H}(E)>0$. Suppose $\mathcal{O}_{\Sigma_{n}}(G)\hookrightarrow E$ is a maximal line subbundle of $E$. Then
    \begin{align*}
        0\leq\hbox{deg}_{H}(\mathcal{O}_{\Sigma_{n}}(G))<{\alpha\over 2}.
    \end{align*}
    Note that $\hbox{deg}_{H}(E)-2\hbox{deg}_{H}(\mathcal{O}_{\Sigma_{n}}(L)) = \alpha$, and if $S_{H}(E)<\alpha$, then $\hbox{deg}_{H}(\mathcal{O}_{\Sigma_{n}}(G))>0$. Suppose that, indeed, $0<\hbox{deg}_{H}(\mathcal{O}_{\Sigma_{n}}(G))<{\alpha\over 2}$. Tensoring the exact sequence \eqref{Teorema 3.4suc} by $\mathcal{O}_{\Sigma_{n}}(-G)$ and we obtain:
    \begin{align}
        0 \longrightarrow \mathcal{O}_{\Sigma_{n}}(-G)\longrightarrow E(-G)\longrightarrow \mathcal{I}_{Z}(f-G) \longrightarrow 0, \label{Example 4}
    \end{align}
    If $f-G$ is an effective divisor with $G=a\sigma+bf$ then $a\leq 0$ and $b\leq 1$.

    Assume that $b=1$. Since $f-G$ is effective, then $(f-G).H = a(\alpha n-\beta)+(1-b)(\alpha) = a(\alpha n-\beta) > 0$, therefore $a<0$, since $\alpha n-\beta<0$ and $G\neq L$. Furthermore, in this case we have that $\text{deg}_{H}(\mathcal{O}_{\Sigma_{n}}(G)) = a(\beta-\alpha n)+\alpha$. Since $\beta>\alpha n+(\alpha-1)$, then $\beta-\alpha n\geq\alpha$, and therefore, since $a<0$, then $a(\beta-\alpha n)+\alpha\leq 0$; that is, $\deg_{H}(\mathcal{O}_{\Sigma_{n}}(G))\leq 0$, which is impossible, and therefore, $b\neq 1$; that is, $b\leq 0$.

    With the above, we have that $0 < \hbox{deg}_{H}(\mathcal{O}_{\Sigma_{n}}(G)) = a(\beta-\alpha n)+b\alpha \leq \alpha b \leq 0$, which is impossible, and therefore, $f-G=-a\sigma+(1-b)f$ cannot be effective, and then $a>0$ or $b>1$. If we assume that $a>0$ and $b>1$, then
    \begin{align*}
        \mu_{H}(\mathcal{O}_{\Sigma_{n}}(G)) = a(\beta-\alpha n)+b(\alpha) > \alpha > {\alpha\over 2} = \mu_{H}(E),
    \end{align*}
    which contradicts the $H$-stability of $E$. So either $a>0$ or $b>1$, but not both.

    Recall that, since $\mathcal{O}_{\Sigma_{n}}(G)$ is a maximal line subbundle of $E$, then there also exists an exact sequence of the form
    \begin{align*}
        0 \longrightarrow \mathcal{O}_{\Sigma_{n}}(G)\longrightarrow E\longrightarrow \mathcal{I}_{Z'}(f-G) \longrightarrow 0,
    \end{align*}
    where $Z'\subset\Sigma_{n}$ is a $0$-dimensional subscheme of length $|Z|=c_{2}-G.(f-G)$. Based on the above and according to the values of $a$ and $b$, the following cases are obtained:

    {\bf Case (a)} Assume that $a>0$, whence $b\leq 1$. Since $\beta-\alpha n\geq\alpha$ and $a>0$, then if $b\in\{0,1\}$ we have $\text{deg}_{H}(\mathcal{O}_{\Sigma_{n}}(G))=a(\beta-\alpha n)+b\alpha\geq\alpha$, which is impossible. Then $b<0$ and therefore $G=a\sigma+bf$ is non-effective. Thus, since $G$ and $f-G$ are not effective, then $h^{0}(E)=0$.

    {\bf Case (b)} Assume that $b>1$, whence $a\leq 0$. If $a=0$, then $\text{deg}_{H}(\mathcal{O}_{\Sigma_{n}}(G))=b\alpha>\alpha$, which is impossible. Then $a<0$ and therefore $G=a\sigma+bf$ is non-effective. Thus, since $G$ and $f-G$ are not effective, then $h^{0}(E)=0$.

    Then it follows that if $E$ is a rank-2 vector bundle on $\Sigma_{n}$ with $c_{1}(E)=f$ such that its maximal line subbundle $\mathcal{O}_{\Sigma_{n}}(G)$ is such that $0<\text{deg}_{H}(\mathcal{O}_{\Sigma_{n}}(G))<{\alpha\over 2}$ then $h^{0}(E)=0$. But from the original sequence \eqref{Teorema 3.4suc} we have that $h^{0}(E)\geq 1$, then $\mathcal{O}_{\Sigma_{n}}(G)$ cannot exist. Then $\mathcal{O}_{\Sigma_{n}}(L)=\mathcal{O}_{\Sigma_{n}}$ is a maximal line subbundle of $E$, and $S_{H}(E)=\alpha$.

    The proof when $a=1$; that is, $H\equiv\sigma+\beta f$ is analogous, but considerably simpler, therefore, we omit the details.

    Let us compute the global sections of the vector bundle $E$. Since $h^{0}(\mathcal{O}_{\Sigma_{n}})=1$ and $h^{1}(\mathcal{O}_{\Sigma_{n}})=0$ then, from the exact sequence in cohomology induced by \eqref{Teorema 3.4suc}, we have that $h^{0}(E)=1+h^{0}(\mathcal{I}_{Z}(f))$.

    Recall that if $|Z|=c_{2}\geq 2=h^{0}(\mathcal{O}_{\Sigma_{n}}(f))$ with $Z$ general, then $h^{0}(\mathcal{I}_{Z}(f))=0$ by Proposition \ref{Proposition 1.27}, so $h^{0}(E)=1$ in this case.

    Now, assume that $c_{2}=1$; that is, $Z=\{p\}$, and note that in the exact sequence in cohomology induced by
    \begin{align*}
        0 \longrightarrow \mathcal{I}_{p}(f)\longrightarrow \mathcal{O}_{\Sigma_{n}}(f)\longrightarrow \mathcal{O}_{p}(f)\longrightarrow 0,
    \end{align*}
    we have $h^{0}(\mathcal{O}_{\Sigma_{n}}(f))=2$, $h^{0}(\mathcal{O}_{p}(f))=1$ and $h^{1}(\mathcal{O}_{\Sigma_{n}}(f))=0$. Therefore $h^{0}(\mathcal{I}_{p}(f))=1$, and then in this case, $h^{0}(E)=2$.

    In summary, we have the following:
    \[
    h^{0}(E) = 
    \begin{cases}
    2 & \text{if } c_{2}=1, \\
    1 & \text{if } c_{2}\geq 2. \\
    \end{cases}
    \]

    \item Assume that $H=H_{r}\equiv \alpha\sigma+\beta_{r}f$, with $\beta_{r}=\alpha(n+1)+r$, where $0<r\leq\alpha$ and $c_{2}\geq 2$. Let $Z\subset\Sigma_{n}$ be a general $0$-dimensional subscheme, and let $L=f$ and $M=\sigma$ be line bundles on $\Sigma_{n}$. Then by CB property there exists a non-trivial extension of the form
    \begin{align}
        0 \longrightarrow \mathcal{O}_{\Sigma_{n}}(f)\longrightarrow E\longrightarrow \mathcal{I}_{Z}(\sigma) \longrightarrow 0, \label{Teorema 3.6suc}
    \end{align}
    where $E$ is locally free and has Chern classes $c_{1}=\sigma+f$ and $c_{2}=1+|Z|$.
    
    The proof that the vector bundle $E$ is $H_{r}$-stable and has Segre invariant $S_{H_{r}}(E)=r$ is completely analogous, in general idea and in the steps followed, to that of Case (3), therefore, we omit the details.

    Let us compute the global sections of $E$. Since $h^{0}(\mathcal{O}_{\Sigma_{n}}(f))=2$ and $h^{1}(\mathcal{O}_{\Sigma_{n}}(f))=0$ then, from the exact sequence in cohomology induced by \eqref{Teorema 3.6suc} we have that $h^{0}(E) = h^{0}(\mathcal{O}_{\Sigma_{n}}(f))+h^{0}(\mathcal{I}_{Z}(\sigma)) = 2+h^{0}(\mathcal{I}_{Z}(\sigma)).$

    For $n=0$, we have $h^{0}(\mathcal{O}_{\Sigma_{0}}(\sigma))=2$, and since $|Z|=c_{2}-1$, if $c_{2}\geq 3$ with $Z$ general, then $|Z|\geq h^{0}(\mathcal{O}_{\Sigma_{0}}(\sigma))$, and so $h^{0}(\mathcal{I}_{Z}(\sigma))=0$ by Proposition \ref{Proposition 1.27}. Now, if $c_{2}=2$, then $Z=\{p\}$, and note that we have the same conditions as Case (3). Analogously, it can be shown that and we have $h^{0}(\mathcal{I}_{p}(\sigma))=1$.

    For $n\geq 1$, we have $h^{0}(\mathcal{O}_{\Sigma_{n}}(\sigma))=1$, and since $|Z|=c_{2}-1$ with $c_{2}\geq 2$ and $Z$ general, then $|Z|\geq h^{0}(\mathcal{O}_{\Sigma_{n}}(\sigma))$, then $h^{0}(\mathcal{I}_{Z}(\sigma))=0$ again by Proposition \ref{Proposition 1.27}.

    In summary we have the following:
    \[
    h^{0}(E) = 
    \begin{cases}
    3 & \text{if } n=0 \text{ and } c_{2}=2, \\
    2 & \text{if } n=0 \text{ and } c_{2}\geq 3, \\
    2 & \text{if } n\geq 1.
    \end{cases}
    \]

    \item Assume that $H=H_{r}\equiv \alpha\sigma+\beta_{r}f$, with $\beta_{r}=\alpha(n+1)+r$, where $0<r\leq\alpha$ and $c_{2}\geq 5$. Let $Z\subset\Sigma_{n}$ be a general $0$-dimensional subscheme, and let $L=0$ and $M=\sigma+f$ be line bundles on $\Sigma_{n}$. Then by CB property there exists a non-trivial extension of the form
    \begin{align}
        0\longrightarrow \mathcal{O}_{\Sigma_{n}}\longrightarrow E\longrightarrow \mathcal{I}_{Z}(\sigma+f)\longrightarrow 0, \label{Teorema 3.7suc}
\end{align}
    where $E$ is locally free and has Chern classes $c_{1}=\sigma+f$ and $c_{2}=|Z|$.

    Now, let us prove that the vector bundle $E$ is $H_{r}$-stable. Assume $\mathcal{O}_{\Sigma_{n}}(J)\hookrightarrow\mathcal{I}_{Z}(\sigma+f)$ where $J=c\sigma+df$. Then $\sigma+f-J=(1-c)\sigma+(1-d)f$ is an effective divisor, and therefore $c\leq 1$ and $d\leq 1$.

    Suppose that $c=1$ and $d=1$, so $J=\sigma+f$. Then $\mathcal{O}_{\Sigma_{n}}(\sigma+f)\hookrightarrow\mathcal{I}_{Z}(\sigma+f)$, that is $\mathcal{O}_{\Sigma_{n}}\hookrightarrow\mathcal{I}_{Z}$, but this is possible if and only if $Z=\varnothing$, which is a contradiction, and therefore, $c\neq 1$ or $d\neq 1$.

    Suppose that $d\neq 1$; that is $d\leq 0$, and then $c\leq 1$. If $c=1$ we have that $J=\sigma+df$, and since $\mathcal{O}_{\Sigma_{n}}(J)\hookrightarrow\mathcal{I}_{Z}(\sigma+f)$, then $\mathcal{O}_{\Sigma_{n}}\hookrightarrow\mathcal{I}_{Z}([1-d]f)$, where $1-d\geq 1$. But, by Remark \ref{Observation 2.18}$, \mathcal{O}_{\Sigma_{n}}\hookrightarrow\mathcal{I}_{Z}([1-d]f)$ exists if and only if there is $D\in|(1-d)f|$ such that $Z\subset D$, but since $Z$ is general, then the above cannot happen, and therefore, $c\neq 1$; that is, $c\leq 0$. Similarly, it can be verified that if we assume that $c\neq 1$, then $d\neq 1$. Thus, $c,d\leq 0$, and then 
    \begin{align*}
        \mu_{H_{r}}(\mathcal{O}_{\Sigma_{n}}(J) = c(\alpha+r)+d(\alpha) \leq 0 < {2\alpha+r\over 2} = \mu_{H_{r}}(E),
    \end{align*}
    and we obtain that $E$ is $H_{r}$-stable, and so $S_{H_{r}}(E)>0$.

    In this case, it is convenient to first compute the global sections of $E$ before its Segre invariant. Since $h^{0}(\mathcal{O}_{\Sigma_{n}})=1$ and $h^{1}(\mathcal{O}_{\Sigma_{n}})=0$, then $h^{0}(E)=1+h^{0}(\mathcal{I}_{Z}(\sigma+f))$, but $|Z|\geq 5>h^{0}(\mathcal{O}_{\Sigma_{n}}(\sigma+f))$ with $Z$ general, therefore $h^{0}(\mathcal{I}_{Z}(\sigma+f))=0$ by Proposition \ref{Proposition 1.27}, and so $h^{0}(E)=1$.

    Now, for the Segre invariant, suppose $\mathcal{O}_{\Sigma_{n}}(G)\hookrightarrow E$ is a maximal line subbundle of $E$. Then
    \begin{align*}
        0\leq\hbox{deg}_{H_{r}}(\mathcal{O}_{\Sigma_{n}}(G))<{2\alpha+r\over 2}.
    \end{align*}
    Note that $\hbox{deg}_{H_{r}}(E)-2\hbox{deg}_{H_{r}}(\mathcal{O}_{\Sigma_{n}}(L)) = 2\alpha+r$, and if $S_{H_{r}}(E)<2\alpha+r$, then $\hbox{deg}_{H_{r}}(\mathcal{O}_{\Sigma_{n}}(G))>0$. Suppose that, indeed, $0<\hbox{deg}_{H_{r}}(\mathcal{O}_{\Sigma_{n}}(G))<{2\alpha+r\over 2}$. Tensoring the exact sequence \eqref{Teorema 3.7suc} by $\mathcal{O}_{\Sigma_{n}}(-G)$ and we obtain
    \begin{align}
        0 \longrightarrow \mathcal{O}_{\Sigma_{n}}(-G)\longrightarrow E(-G)\longrightarrow \mathcal{I}_{Z}(\sigma+f-G) \longrightarrow 0, \label{Example 2.30.2}
    \end{align}
    Suppose $\sigma+f-G$ is an effective divisor with $G=a\sigma+bf$, and then $a\leq 1$ and $b\leq 1$. Since $H_{r}$ is ample and $\sigma+f-G$ is effective, then $a\neq 1$ or $b\neq 1$ because
    \begin{align*}
        (\sigma+f-G).H_{r} = ([1-a]\sigma+[1-b]f).(\alpha\sigma+\beta_{r}f) = (1-a)(\alpha+r)+(1-b)(\alpha) > 0.
    \end{align*}

    Suppose that $b=1$. If $a=0$, then $G=f$ and since $\mathcal{O}_{\Sigma_{n}}(G)\subset E$ is maximal then there exists a short exact sequence of type
    \begin{align}
        0 \longrightarrow \mathcal{O}_{\Sigma_{n}}(f)\longrightarrow E\longrightarrow \mathcal{I}_{Z'}(\sigma)\longrightarrow 0, \label{Example 2.30.3}
    \end{align}
    with $Z'$ a $0$-dimensional subscheme of length $|Z'|=c_{2}-1$ and therefore $Z'\neq Z$. From this sequence we have that $h^{0}(E)\geq h^{0}(\mathcal{O}_{\Sigma_{n}}(f))=2$, which is a contradiction, and therefore $a<0$. Then we have $\mu_{H_{r}}(\mathcal{O}_{\Sigma_{n}}(G))=\alpha(a+1)+ra<0$, which is impossible and so $b\neq 1$.

    Now, suppose that $a=1$. If $b=0$, then $G=\sigma$ and we have
    \begin{align*}
        \mu_{H_{r}}(\mathcal{O}_{\Sigma_{n}}(G)) = (\sigma).(\alpha\sigma+\beta_{r}f) = \alpha+r > {2\alpha+r\over 2} = \mu_{H_{r}}(E),
    \end{align*}
    which contradicts the $H_{r}$-stability of $E$, and therefore $b<0$.

    Suppose $b=-1$, so $G=\sigma-f$. Then $E$ fits into a short exact
    sequence of the form
    \begin{align*}
    0 \longrightarrow \mathcal{O}_{\Sigma_{n}}(\sigma-f)\longrightarrow E\longrightarrow \mathcal{I}_{Z'}(2f)\longrightarrow 0,
    \end{align*}
    with $Z'$ a general $0$-dimensional subscheme with $Z'\neq Z$ and $|Z'|=c_{2}-2$. Since $|Z'|\geq 3$ with $Z'$ general, and $h^{0}(\mathcal{O}_{\Sigma_{n}}(2f))=3$, then $h^{0}(\mathcal{I}_{Z'}(2f))=0$ again by Proposition \ref{Proposition 1.27}. Furthermore, as $h^{0}(\mathcal{O}_{\Sigma_{n}}(\sigma-f))=0$, then $h^{0}(E)=0$, but from the original sequence \ref{Teorema 3.7suc} we have that $h^{0}(E)\geq 1$, which is a contradiction and therefore, $b<-1$. It follows that
    \begin{align*}
        \deg_{H_{r}}(\mathcal{O}_{\Sigma_{n}}(G)) = \alpha(b+1)+r \leq r-\alpha \leq 0, 
    \end{align*}
    which is a contradiction, and therefore $a<1$. Thus, in general, it must be $a\leq 0$ and $b\leq 0$, but since $G\neq 0$, then necessarily $a<0$ or $b<0$. However, this implies that $\text{deg}_{H_{r}}(\mathcal{O}_{\Sigma_{n}}(G))<0$ and therefore $\mathcal{O}_{\Sigma_{n}}(G)$ is not maximal, and so $\mathcal{O}_{\Sigma_{n}}$ is a maximal line subbundle of $E$; moreover, for this reason we have that $S_{H_{r}}(E)=2\alpha+r$ with $0<r\leq\alpha$. 
\end{enumerate}
\end{proof}

\begin{remark}
Note that in each case of Theorem \ref{Theorem 3.4}, the Segre invariant, which arises from the line subbundles involved in the extensions for which the vector bundles $E$ are given, imposes restrictions on the number of global sections of each vector bundle. In addition to this, the following are some important observations:
\begin{itemize}[leftmargin=*]
    \item Consider the conditions of case (3) of Theorem \ref{Theorem 3.4} in Theorem \ref{Theorem 1.2 Costa}. We have
    \begin{align*}
        1\leq k<{1\over2\alpha}[\beta-\alpha(e-m+2g-2)]=1+{t\over 2\alpha}+{m\over 2},
    \end{align*}
    and since $t>0$ it always follows that
    \begin{align*}
        1<1+{t\over 2\alpha}+{m\over 2}.
    \end{align*}
    Thus, for our case $k=1$ is the only option. Therefore, Theorem \ref{Theorem 1.2 Costa} implies that for $c_{2}>>0$ then $W^{1}_{H_{t}}(2;\sigma,c_{2})\neq\varnothing$ whenever $M_{H_{t}}(2;\sigma,c_{2})\neq\varnothing$; however, the case (3) of Theorem \ref{Theorem 3.4} ensures that $W^{2}_{H_{t}}(2;\sigma,1)\neq\varnothing$ on $\Sigma_{0}$, or even without imposing extra conditions on $c_{2}$ (unless it is positive), the Brill-Noether locus $W^{1}_{H_{t}}(2;\sigma,c_{2})\neq\varnothing$ on $\Sigma_{n}$ for $n\geq 1$.

    \item Theorem \cite[Theorem 4.8]{Costa3} shows that the varieties $W^{1}_{H}(2;f,c_{2})$ are non-empty with $H\equiv \sigma+\beta f$; a result that inspired the case (5) of Theorem \ref{Theorem 3.4} focused on surfaces $\Sigma_{n}$; however, case (4) provides a generalization of the result of Costa and Macías-Tarrío on Hirzebruch surfaces for a larger collection of ample divisors.
\end{itemize}
\end{remark}

\section{Applications to Brill-Noether Theory}

The goal of this section is to determine the values of $k$ that guarantee that the Brill-Noether locus $W^{k}_{H}(2;c_{1},c_{2})$ is non-empty. Furthermore, we will show that these determinant variety are smooth and of the expected dimension $\rho^{k}_{H}$.

\begin{proposition}\label{Proposition 4.3.1}
Let $H\in\text{Pic}(\Sigma_{n})$ be an ample divisor and $c_{1}\in\text{Pic}(\Sigma_{n})$ effective. If $E\in M_{\Sigma_{n},H}(2;c_{1},c_{2})$, then $h^{2}(E)=0$.
\end{proposition}

\begin{proof}
Note that $\mu_{H}(E)\geq 0$ for every ample divisor $H$ on $\Sigma_{n}$ if $c_{1}\in\{0,\sigma,f,\sigma+f\}$. By Serre duality, $H^{2}(E)\cong H^{0}(E^{*}\otimes\mathcal{O}_{\Sigma_{n}}(K_{\Sigma_{n}}))$, where by Remark \ref{Lemma 4.2}, $E^{*}\otimes\mathcal{O}_{\Sigma_{n}}(K_{\Sigma_{n}})$ is $H$-stable. Furthermore, $ \mu_{H}(E^{*}\otimes\mathcal{O}_{\Sigma_{n}}(K_{\Sigma_{n}})) = -\mu_{H}(E)+\mu_{H}(\mathcal{O}_{\Sigma_{n}}(K_{\Sigma_{n}}))$, where $\mu_{H}(\mathcal{O}_{\Sigma_{n}}(K_{\Sigma_{n}}))<0$, so $\mu_{H}(\mathcal{O}_{\Sigma_{n}}(K_{\Sigma_{n}}))<\mu_{H}(E)$; that is, $\mu_{H}(E^{*}\otimes\mathcal{O}_{\Sigma_{n}}(K_{\Sigma_{n}}))<0$.

Suppose there exists $s\in H^{0}(E^{*}\otimes\mathcal{O}_{\Sigma_{n}}(K_{\Sigma_{n}}))$, which corresponds to a bundle homomorphism $\phi:\mathcal{O}_{\Sigma_{n}}\longrightarrow E^{*}\otimes\mathcal{O}_{\Sigma_{n}}(K_{\Sigma_{n}})$. The image $S:=\text{Im}(\phi)$ is a torsion-free subsheaf of $E^{*}\otimes\mathcal{O}_{\Sigma_{n}}(K_{\Sigma_{n}})$ of rank $1$, so $S\cong\mathcal{I}_{Z}(L)$ with $L$ a line divisor on $\Sigma_{n}$ and $\mathcal{I}_{Z}$ the ideal of a $0$-dimensional subscheme $Z\subset\Sigma_{n}$. Furthermore, $h^{0}(S)>0$ because the global section $s$ has image in $S$. Since there exists an inclusion $\mathcal{I}_{Z}(L)\hookrightarrow L$, if $h^{0}(\mathcal{I}_{Z}(L))>0$ then $h^{0}(L)>0$, so $L$ is effective, and then $\mu_{H}(S)=\mu_{H}(\mathcal{I}_{Z}(L))=L.H\geq 0$.

Now, since $E^{*}\otimes\mathcal{O}_{\Sigma_{n}}(K_{\Sigma_{n}})$ is $H$-stable and $S$ is a subsheaf of $E^{*}\otimes\mathcal{O}_{\Sigma_{n}}(K_{\Sigma_{n}})$ of rank $1$, then $0\leq\mu_{H}(S)<\mu_{H}(E^{*}\otimes\mathcal{O}_{\Sigma_{n}}(K_{\Sigma_{n}}))<0$, which is a contradiction. Therefore, $h^{0}(E^{*}\otimes\mathcal{O}_{\Sigma_{n}}(K_{\Sigma_{n}}))=0$, and so $h^{2}(E)=0$.
\end{proof}

\begin{proposition}\label{Proposition 4.3.2}
Let $c_{1}\in\{\sigma,f,\sigma+f\}$ and $D_{c_{1}}\in\text{Pic}(\Sigma_{n})$ which depends on $c_{1}$ as follows:
\begin{enumerate}
    \item For $c_{1}=\sigma$, $D_{c_{1}}=0$,
    \item For $c_{1}=f$, $D_{c_{1}}=0$,
    \item For $c_{1}=\sigma+f$, $D_{c_{1}}=f$ or $D_{c_{1}}=0$.
\end{enumerate}
Let $E\in M_{\Sigma_{n},H}(2;c_{1},c_{2})$. Then $h^{2}(E\otimes E^{*})=0$ if $E$ fits into a short exact sequence of type
\begin{align}
0 \longrightarrow \mathcal{O}_{\Sigma_{n}}(D_{c_{1}})\longrightarrow E\longrightarrow \mathcal{I}_{Z}(c_{1}-D_{c_{1}})\longrightarrow 0. \label{squence Proposition 4.3.2}
\end{align}
where $Z$ is any $0$-dimensional subscheme of length $|Z|=c_{2}-D_{c_{1}}.(c_{1}-D_{c_{1}})$.
\end{proposition}

\begin{proof}
First, we show that $h^{2}(E^{*})=0$. If we tensor the short exact sequence
\eqref{squence Proposition 4.3.2} by $\mathcal{O}_{\Sigma_{n}}(K_{\Sigma_{n}})$ and we take cohomology, we get the long exact sequence
\begin{align*}
    0 \longrightarrow H^{0}(\mathcal{O}_{\Sigma_{n}}(D_{c_{1}}+K_{\Sigma_{n}}))\longrightarrow H^{0}(E(K_{\Sigma_{n}}))\longrightarrow H^{0}(\mathcal{I}_{Z}(c_{1}-D_{c_{1}}+K_{\Sigma_{n}}))\longrightarrow...
\end{align*}
For every choice of $c_{1}$ both $D_{c_{1}}+K_{\Sigma_{n}}$ and $c_{1}-D_{c_{1}}+K_{\Sigma_{n}}$ are not effective, then $h^{0}(\mathcal{O}_{\Sigma_{n}}(D_{c_{1}}+K_{\Sigma_{n}}))=0$ and $h^{0}(\mathcal{I}_{Z}(c_{1}-D_{c_{1}}+K_{\Sigma_{n}}))=0$, and so, $h^{2}(E^{*})=h^{0}(E(K_{\Sigma_{n}}))=0$, by Serre duality.

We will see that $h^{2}(E^{*}(D_{c_{1}}))=0$ and $h^{2}(E^{*}\otimes\mathcal{I}_{Z}(c_{1}-D_{c_{1}}))=0$. First, since $K_{\Sigma_{n}}$ and $K_{\Sigma_{n}}+c_{1}-2D_{c_{1}}$ are not effective, when we tensor the short exact sequence \eqref{squence Proposition 4.3.2} by $\mathcal{O}_{\Sigma_{n}}(K_{\Sigma_{n}}-D_{c_{1}})$, we get $h^{0}(\mathcal{O}_{\Sigma_{n}}(K_{\Sigma_{n}}))=0$ and $h^{0}(\mathcal{I}_{Z}(K_{\Sigma_{n}}+c_{1}-2D_{c_{1}}))=0$ and then $h^{0}(E(K_{\Sigma_{n}}-D_{c_{1}}))=h^{2}(E^{*}(D_{c_{1}}))=0$. Now, since $h^{i}(E^{*}\otimes\mathcal{O}_{Z}(c_{1}-D_{c_{1}}))=0$ for $i\in\{1,2\}$, if we tensor the short exact sequence
\begin{align*}
    0 \longrightarrow \mathcal{I}_{Z}\longrightarrow \mathcal{O}_{\Sigma_{n}}\longrightarrow\mathcal{O}_{Z}\longrightarrow 0,
\end{align*}
by $E^{*}(c_{1}-D_{c_{1}})$, we have that, $h^{2}(E^{*}\otimes\mathcal{I}_{Z}(c_{1}-D_{c_{1}}))=h^{2}(E^{*}(c_{1}-D_{c_{1}}))$. Since for any value of $c_{1}$ it is holds that $2D_{c_{1}}+K_{\Sigma_{n}}-c_{1}$ and $K_{\Sigma_{n}}$ are not effective, when we tensor the exact sequence (\ref{squence Proposition 4.3.2}) by $\mathcal{O}_{\Sigma_{n}}(D_{c_{1}}+K_{\Sigma_{n}}-c_{1})$, we have that $h^{0}(E(D_{c_{1}}+K_{\Sigma_{n}}-c_{1}))=h^{2}(E^{*}(c_{1}-D_{c_{1}}))=h^{2}(E^{*}\otimes\mathcal{I}_{Z}(c_{1}-D_{c_{1}}))=0$.

Finally, since $h^{2}(E^{*}(D_{c_{1}}))=0$ and $h^{2}(E^{*}\otimes\mathcal{I}_{Z}(c_{1}-D_{c_{1}}))=0$, when we tensor (\eqref{squence Proposition 4.3.2}) by $E^{*}$ we deduce that $h^{2}(E\otimes E^{*})=0$. Hence, $E$ is a smooth point of $M_{\Sigma_{n},H}(2;c_{1},c_{2})$.
\end{proof}

\begin{proposition}\label{Proposition 4.3.3}
Let $c_{1}\in\{\sigma,f,\sigma+f\}$ and $E\in M_{\Sigma_{n},H}(2;c_{1},c_{2})$ given by a short exact sequence of type
\begin{align}
0 \longrightarrow \mathcal{O}_{\Sigma_{n}}\longrightarrow E\longrightarrow \mathcal{I}_{Z}(c_{1})\longrightarrow 0, \label{squence Proposition 4.3.3}
\end{align}
where $Z$ is any $0$-dimensional subscheme of length $|Z|=c_{2}$. If $h^{0}(E)=1$, then the map
\begin{align*}
\mu_{E}:H^{0}(E)\otimes H^{1}(E^{*}\otimes\mathcal{O}_{\Sigma_{n}}(K_{\Sigma_{n}}))\longrightarrow H^{1}(E\otimes E^{*}\otimes\mathcal{O}_{\Sigma_{n}}(K_{\Sigma_{n}}))
\end{align*}
is injective.
\end{proposition}

\begin{proof}
In this case the injectivity of $\mu_{E}$ is equivalent to proving the injectivity of
\begin{align*}
    \tilde{\mu_{E}}:H^{1}(E^{*}\otimes\mathcal{O}_{\Sigma_{n}}(K_{\Sigma_{n}}))\longrightarrow H^{1}(E\otimes E^{*}\otimes\mathcal{O}_{\Sigma_{n}}(K_{\Sigma_{n}})).
\end{align*}
Since $h^{2}(\mathcal{I}_{Z})=h^{2}(\mathcal{O}_{\Sigma_{n}})=0$ and by Serre duality $h^{2}(\mathcal{O}_{\Sigma_{n}}(-c_{1}))=h^{0}(\mathcal{O}_{\Sigma_{n}}(c_{1}+K_{\Sigma_{n}}))=0$, if we tensor the exact sequence (\ref{squence Proposition 4.3.3}) by $\mathcal{O}_{\Sigma_{n}}(-c_{1})$, we deduce that $h^{2}(E(-c_{1}))=0$ and again by Serre duality $h^{0}(E^{*}(K_{\Sigma_{n}}+c_{1}))=0$. Furthermore, if $h^{0}(E^{*}(K_{\Sigma_{n}}+c_{1}))=0$, we have from the short exact sequence
\begin{align*}
0\longrightarrow\mathcal{I}_{Z}(E^{*}(K_{\Sigma_{n}}+c_{1}))\longrightarrow\mathcal{O}_{\Sigma_{n}}(E^{*}(K_{\Sigma_{n}}+c_{1}))\longrightarrow\mathcal{O}_{Z}(E^{*}(K_{\Sigma_{n}}+c_{1}))\longrightarrow 0,
\end{align*}
that $h^{0}(\mathcal{I}_{Z}(E^{*}(K_{\Sigma_{n}}+c_{1})))=0$. If we tensor the exact sequence (\ref{squence Proposition 4.3.3}) by $E^{*}\otimes\mathcal{O}_{\Sigma_{n}}(K_{\Sigma_{n}})$ and we take cohomology, we get 
\begin{align*}
0\longrightarrow H^{1}(E^{*}(K_{\Sigma_{n}}))\stackrel{\tilde{\mu_{E}}}\longrightarrow H^{1}(E\otimes E^{*}(K_{\Sigma_{n}}))\longrightarrow...
\end{align*} 
which implies that $\tilde{\mu_{E}}$ is injective, and so, $\mu_{E}$ is injective. 
\end{proof}

\begin{theorem}\label{Theorem 4.6}
Let $E\in M_{\Sigma_{n},H}(2;c_{1},c_{2})$ with $c_{1}=c\sigma+df$ effective and $c_{2}>0$ such that $E$ lies in an extension of type
\begin{align}
    0\longrightarrow \mathcal{O}_{\Sigma_{n}}(D)\longrightarrow E \longrightarrow \mathcal{I}_{Z}(c_{1}-D)\longrightarrow 0, \label{Example-muE}
\end{align}
where $Z$ is a general $0$-dimensional subscheme of length $|Z|=c_{2}-D.(c_{1}-D)$ such that $|Z|=h^{0}(\mathcal{O}_{\Sigma_{n}}(c_{1}))$ where $D=a\sigma+bf$ is a base-point-free divisor on $\Sigma_{n}$ such that $h^{0}(\mathcal{O}_{\Sigma_{n}}(D))\geq 2$, $h^{1}(\mathcal{O}_{\Sigma_{n}}(2D))=0$ and $2(a+1)>c$ or $2(b+1)>d$. Then, if $h^{1}(\mathcal{O}_{\Sigma_{n}}(c_{1})))=0$, the map
\begin{align*}
    \mu_{E}:H^{0}(E)\otimes H^{1}(E^{*}\otimes\mathcal{O}_{\Sigma_{n}}(K_{\Sigma_{n}}))\longrightarrow H^{1}(E\otimes E^{*}\otimes\mathcal{O}_{\Sigma_{n}}(K_{\Sigma_{n}}))
\end{align*}
is injective.
\end{theorem}

\begin{proof}
First, note that on $\Sigma_{n}$ with $n\geq 1$, the only effective divisors with exactly a unique global section are divisors of the form $L=r\sigma$ with $r\geq 0$ and on $\Sigma_{0}$ necessarily $L=0$, then, since $D$ is such that $h^{0}(D)\geq 2$, it is necessary that $D\neq 0$ for every $n\geq 0$.

Note that $K_{\Sigma_{n}}+c_{1}-2D$ is non-effective because $c<2(a+1)$ or $d<2(b+1)$, when we tensor (\ref{Example-muE}) by $\mathcal{O}_{\Sigma_{n}}(D-c_{1})$ and take cohomology, we get $H^{2}(\mathcal{O}_{\Sigma_{n}}(2D-c_{1}))\cong H^{0}(\mathcal{O}_{\Sigma_{n}}(K_{\Sigma_{n}}+c_{1}-2D))=\{0\}$, and therefore, $H^{2}(E(D-c_{1}))\cong H^{2}(\mathcal{I}_{Z})=\{0\}$.

Moreover, $H^{0}(E^{*}\otimes\mathcal{I}_{Z}(c_{1}+K_{\Sigma_{n}}-D))\cong H^{2}(E(D-c_{1}))$ by Serre duality, then if we tensor (\ref{Example-muE}) by $E^{*}\otimes\mathcal{O}_{\Sigma_{n}}(K_{\Sigma_{n}})$ and take cohomology, we have
\begin{align*}
    \varphi: H^{1}(E^{*}\otimes\mathcal{O}_{\Sigma_{n}} (K_{\Sigma_{n}})\otimes\mathcal{O}_{\Sigma_{n}}(D))\longrightarrow H^{1}(E\otimes E^{*}\otimes\mathcal{O}_{\Sigma_{n}}(K_{\Sigma_{n}})
\end{align*}
is injective.

We have the commutative diagram (\ref{diagrama-muE}), where $\tau$ comes from the cup product.

\begin{figure}[h]
\centering
\[
\begin{tikzcd}[column sep=5em, row sep=5em]
H^{0}(E)\otimes H^{1}(E^{*}\otimes \mathcal{O}_{\Sigma_{n}}(K_{\Sigma_{n}}))
\arrow[r, "\tau"']
\arrow[dr, "\mu_{E}"']
&
H^{1}(E^{*}\otimes\mathcal{O}_{\Sigma_{n}}(K_{\Sigma_{n}})\otimes\mathcal{O}_{\Sigma_{n}}(D))
\arrow[d, "\varphi"]
\\
&
\hspace{-3em}
H^{1}(E\otimes E^{*}\otimes \mathcal{O}_{\Sigma_{n}}(K_{\Sigma_{n}}))
\end{tikzcd}
\]
\caption{Commutative diagram of $\mu_{E}$.}
\label{diagrama-muE}
\end{figure}

Since $|Z|=h^{0}(\mathcal{O}_{\Sigma_{n}}(c_{1}))\geq h^{0}(\mathcal{O}_{\Sigma_{n}}(c_{1}-D))$ by cohomology, then there is an isomorphism $\phi:H^{0}(E)\longrightarrow H^{0}(\mathcal{O}_{\Sigma_{n}}(D))$ by Proposition \ref{Proposition 1.27}, and then $H^{0}(E)\otimes\mathcal{O}_{\Sigma_{n}}\cong H^{0}(D)\otimes\mathcal{O}_{\Sigma_{n}}$. Moreover, since $D$ is free of base points, then the mapping
\[
\begin{aligned}
ev_{D}:H^{0}(D)\otimes\mathcal{O}_{\Sigma_{n}} &\longrightarrow \mathcal{O}_{\Sigma_{n}}(D)\\
(t,p) &\longmapsto t(p)
\end{aligned}
\]
is surjective, and therefore so is
\[
\begin{aligned}
ev:H^{0}(E)\otimes\mathcal{O}_{\Sigma_{n}} &\longrightarrow \mathcal{O}_{\Sigma_{n}}(D)\\
(s,p) &\longmapsto t_{s}(p),
\end{aligned}
\]
where $t_{s}=\phi(s)$. Since $H^{0}(E)\otimes\mathcal{O}_{\Sigma_{n}}$ is a trivial sheaf of rank $r\geq h^{0}(\mathcal{O}_{\Sigma_{n}}(D))\geq 2$, then $\text{det}(H^{0}(E)\otimes\mathcal{O}_{\Sigma_{n}})=\mathcal{O}_{\Sigma_{n}}$; and on the other hand,
\begin{align*}
    \text{det}(H^{0}(E)\otimes\mathcal{O}_{\Sigma_{n}}) = \text{det}(\mathcal{O}_{\Sigma_{n}}(D))\otimes\text{det}(\text{Ker}(ev)) = \mathcal{O}_{\Sigma_{n}}(D)\otimes\text{det}(\text{Ker}(ev)),
\end{align*}
and so, $\text{det}(\text{Ker}(ev))=\mathcal{O}_{\Sigma_{n}}(-D)$. Thus, we have the following exact sequence
\begin{align}
    0\longrightarrow\mathcal{O}_{\Sigma_{n}}(-D)\longrightarrow H^{0}(E)\otimes\mathcal{O}_{\Sigma_{n}}\longrightarrow\mathcal{O}_{\Sigma_{n}}(D)\longrightarrow 0. \label{Example-muE2}
\end{align}
Note that $H^{0}(H^{0}(E)\otimes E^{*}\otimes\mathcal{O}_{\Sigma_{n}}(K_{\Sigma_{n}}))=H^{0}(E)\otimes H^{0}(E^{*}\otimes\mathcal{O}_{\Sigma_{n}}(K_{\Sigma_{n}}))=H^{0}(E)\otimes H^{2}(E)=0$ by Proposition \ref{Proposition 4.3.1}. Since $H^{1}(H^{0}(E)\otimes E^{*}(K_{\Sigma_{n}}))=H^{0}(E)\otimes H^{1}(E^{*}\otimes\mathcal{O}_{\Sigma_{n}}(K_{\Sigma_{n}}))$, if we tensor (\ref{Example-muE2}) by $E^{*}\otimes\mathcal{O}_{\Sigma_{n}}(K_{\Sigma_{n}})$ and take cohomology, we have that $\tau$ is injective if $h^{1}(E^{*}(K_{\Sigma_{n}}-D))=h^{1}(E(D))=0$.

Now, since $h^{1}(\mathcal{O}_{\Sigma_{n}}(2D))=0$ and $h^{2}(\mathcal{O}_{\Sigma_{n}}(2D))=h^{0}(\mathcal{O}_{\Sigma_{n}}(K_{\Sigma_{n}}-2D))=0$ because $D$ is effective, then if we tensor (\ref{Example-muE}) by $\mathcal{O}_{\Sigma_{n}}(D)$ and take cohomology, we have $h^{1}(E(D))=h^{1}(\mathcal{I}_{Z}(c_{1}))$. We tensor by $\mathcal{O}_{\Sigma_{n}}(c_{1})$ the exact sequence
\begin{align*}
0\longrightarrow\mathcal{I}_{Z}\longrightarrow\mathcal{O}_{\Sigma_{n}}\longrightarrow\mathcal{O}_{Z}\longrightarrow 0,
\end{align*}
and since $|Z|=h^{0}(\mathcal{O}_{\Sigma_{n}}(c_{1}))$ with $Z$ general, then by Proposition \ref{Proposition 1.27}, $h^{0}(\mathcal{I}_{Z}(c_{1}))=0$. On the other hand, since $h^{i}(\mathcal{O}_{Z}(c_{1}))=0$ for $i\in\{1,2\}$, then $h^{2}(\mathcal{I}_{Z}(c_{1}))=h^{2}(\mathcal{O}_{\Sigma_{n}}(c_{1}))=h^{0}(\mathcal{O}_{\Sigma_{n}}(K_{\Sigma_{n}}-c_{1}))=0$ because $c_{1}$ is effective. Therefore, we have the following:
\begin{align*}
   h^{1}(\mathcal{I}_{Z}(c_{1})) & = -\chi(\mathcal{I}_{Z}(c_{1})) \\
   & = |Z|-\chi(\mathcal{O}_{\Sigma_{n}}(c_{1})) \\
   & = |Z|-h^{0}(\mathcal{O}_{\Sigma_{n}}(c_{1}))+h^{1}(\mathcal{O}_{\Sigma_{n}}(c_{1}))-h^{2}(\mathcal{O}_{\Sigma_{n}}(c_{1})) \\
   & = h^{1}(\mathcal{O}_{\Sigma_{n}}(c_{1}))-h^{2}(\mathcal{O}_{\Sigma_{n}}(c_{1})),
\end{align*}
but $h^{2}(\mathcal{O}_{\Sigma_{n}}(c_{1}))=h^{0}(\mathcal{O}_{\Sigma_{n}}(K_{\Sigma_{n}}-c_{1}))=0$, then $h^{1}(\mathcal{I}_{Z}(c_{1}))= h^{1}(\mathcal{O}_{\Sigma_{n}}(c_{1}))$. Therefore, if $h^{1}(\mathcal{O}_{\Sigma_{n}}(c_{1}))=0$, then $\tau$ is injective, and by the commutative diagram (\ref{diagrama-muE}), then $\mu_{E}=\varphi\circ\tau$ is also injective.
\end{proof}

\begin{remark}\label{Remark 4.5}
Let $E\in M_{\Sigma_{n},H}(2;c_{1},c_{2})$ such that $h^{0}(E)>0$, and let $0\neq s\in H^{0}(E)$. Denote by $Y$ the scheme of zeros of $s$, and let $D=a\sigma+bf$ be the maximal effective divisor contained in $Y$. Then, $s$ can be regarded as a section of $E(-D)$ and its scheme of zeros has codimension greater than or equal to two. Thus, we have a short exact sequence
\begin{align*}
    0\longrightarrow \mathcal O_{\Sigma_n}(D)\longrightarrow E\longrightarrow \mathcal{I}_{Z}(c_{1}-D)\longrightarrow 0,
\end{align*}
where $Z$ is a locally complete intersection $0$-cycle of length $|Z|=c_{2}-D.(c_{1}-D)$.
\end{remark}

\begin{proposition}\label{Proposition 4.8}
Let $\Sigma_{n}$ be a Hirzebruch surface and $H_{t}\equiv \alpha\sigma+\beta_{t}f$ an ample divisor on $\Sigma_{n}$ with $\beta_{t}=\alpha n+t$ where $0<t\leq\alpha$. Then, $W_{H_{t}}^{1}(2;\sigma,c_{2})\setminus W_{H_{t}}^{2}(2;\sigma,c_{2})$ is smooth and irreducible of the expected dimension, namely $\rho_{H_{t}}^{1}=3c_{2}-1$ if $n\geq 1$ and $c_{2}>0$ or $n=0$ and $c_{2}\geq 2$.
\end{proposition}

\begin{proof}
Let us take a vector bundle $E\in W_{H_{t}}^{1}(2;\sigma,c_{2})\setminus W_{H_{t}}^{2}(2;\sigma,c_{2})$. Since $h^{0}(E)=1$, by Remark \ref{Remark 4.5} there exists an exact sequence 
\begin{align*}
    0 \longrightarrow \mathcal{O}_{\Sigma_{n}}(D)\longrightarrow E\longrightarrow \mathcal{I}_{Z}(\sigma-D)\longrightarrow 0,
\end{align*}
where $Z$ is a locally complete intersection $0$-cycle of length $|Z|=c_{2}-D.(\sigma -D)$.

Since $D=a\sigma+bf$ is effective, $a,b\geq 0$. We show that $D=0$. Assume that $D\neq 0$. Since $D$ is effective, $D.L>0$ for any ample divisor $L$ on $\Sigma_{n}$. On the other hand, since $E$ is $H_{t}$-stable, we have that $\mu_{H_{t}}(\mathcal{O}_{\Sigma_{n}}(D))<\mu_{H_{t}}(E)$. Putting it all altogether, we get
\begin{align*}
    0<\mu_{H_{t}}(\mathcal{O}_{\Sigma_{n}}(D))=D.H_{t} < {t\over 2} =\mu_{H_{t}}(E),
\end{align*}
but, we have
\begin{align*}
    \mu_{H_{t}}(\mathcal{O}_{\Sigma_{n}}(D))=at+b\alpha<{t\over 2}=\mu_{H_{t}}(E) \Longleftrightarrow 2at+2b\alpha<t.
\end{align*}
If $a>0$, then $2at>t$, so $2at+2b\alpha>t$, which is impossible. Then $a=0$ and $b>0$. In this case, we have $2b\alpha<t$, which is a contradiction because $0<t\leq\alpha$. Hence, $D=0$. Therefore, any vector bundle $E\in W_{H_{t}}^{1}(2;\sigma,c_{2})\setminus W_{H_{t}}^{2}(2;\sigma,c_{2})$ sits in a short exact sequence
\begin{align}
    0 \longrightarrow \mathcal{O}_{\Sigma_{n}}\longrightarrow E\longrightarrow \mathcal{I}_{Z}(\sigma)\longrightarrow 0. \label{squence 4.3}
\end{align}
By Proposition \ref{Proposition 4.3.2} we have that $h^{2}(E^{*}\otimes E)=0$ for any vector bundle $E$ in $W_{H_{t}}^{1}(2;\sigma,c_{2})\setminus W_{H_{t}}^{2}(2;\sigma,c_{2})$, and hence, $E$ is a smooth point of $M_{H_{t}}(2;\sigma,c_{2})$. Furthermore, by Proposition \ref{Proposition 4.3.3} we have that $\mu_{E}$ is injective. Therefore, by Corollary \ref{Corollary 4.1}, $W_{H_{t}}^{1}(2;\sigma,c_{2})\setminus W_{H_{t}}^{2}(2;\sigma,c_{2})$ is smooth and of the expected dimension at $E$, which is
\begin{align*}
    \rho_{H_{t}}^{1}(2;\sigma,c_{2})=3c_{2}-1
\end{align*}
for all $E\in W_{H_{t}}^{1}(2;\sigma,c_{2})\setminus W_{H_{t}}^{2}(2;\sigma,c_{2})$.

It remains only to establish the irreducibility of $W_{H_{t}}^{1}(2;\sigma,c_{2})\setminus W_{H_{t}}^{2}(2;\sigma,c_{2})$. As established above, any $E\in W_{H_{t}}^{1}(2;\sigma,c_{2})\setminus W_{H_{t}}^{2}(2;\sigma,c_{2})$ sits in an extension of type \ref{squence 4.3}. Denote by $S$ the family of rank two vector bundles $E$ on $\Sigma_{n}$ given by an extension of type \ref{squence 4.3}.

Let $X:=Hilb^{c_{2}}(\Sigma_{n})\times\Sigma_{n}$, $\pi_{1}:X\longrightarrow Hilb^{c_{2}}(\Sigma_{n})$ be the natural projection to the first coordinate. Let $\mathcal{F}:=\mathcal{H}om_{\mathcal{O}_{X}}(\mathcal{I}_{\mathcal{Z}}(\sigma),\mathcal{O}_{X})$ where  $\mathcal{O}_{X}$ is the structural sheaf of the total space, $\mathcal{Z}\subset X$ is the universal subscheme, $\mathcal{I}_{\mathcal{Z}}$ is the corresponding universal sheaf of ideals and $\mathcal{I}_{\mathcal{Z}}(\sigma):=\mathcal{I}_{\mathcal{Z}}\otimes\pi_{2}^{*}\mathcal{O}_{\Sigma_{n}}(\sigma)$ where $\pi_{2}:X\longrightarrow \Sigma_{n}$ be the natural projection to the second coordinate (see \cite{Timofeeva} for more details above the universal subscheme). Note that, in the sense of the Semicontinuity Theorem (see \cite[Theorem 3.12.8]{Hartshorne}), for $i=1$ we get
\begin{align*}
    h^{1}(Z,\mathcal{F}) & = \text{dim }H^{1}(X_{Z};\mathcal{F}_{Z}) \\
    & = \text{dim }H^{1}(\Sigma_{n};\mathcal{H}om_{\mathcal{O}_{\Sigma_{n}}}(\mathcal{I}_{Z}(\sigma),\mathcal{O}_{\Sigma_{n}}) \\
    & = \text{dim } Ext^{1}(\mathcal{I}_{Z}(\sigma),\mathcal{O}_{\Sigma_{n}}).
\end{align*}
Now, by Serre duality, we get $\text{ext}^{1}(\mathcal{I}_{Z}(\sigma), \mathcal{O}_{\Sigma_{n}})=h^{1}(\mathcal{I}_{Z}(\sigma+K_{\Sigma_{n}}))$. From the short exact sequence 
\begin{align*}
    0\longrightarrow \mathcal{I}_{Z} \longrightarrow \mathcal{O}_{\Sigma_{n}}\longrightarrow \mathcal{O}_{Z}\longrightarrow 0,
\end{align*}
by Lemma \ref{Proposition 1.27} and Serre duality, we have $h^{0}(\mathcal{I}_{Z}(\sigma+K_{\Sigma_{n}}))\leq h^{0}(\mathcal{O}_{\Sigma_{n}}(\sigma+K_{\Sigma_{n}}))=0$, and also $h^{2}(\mathcal{I}_{Z}(\sigma+K_{\Sigma_{n}}))=h^{2}(\mathcal{O}_{\Sigma_{n}}(\sigma+K_{\Sigma_{n}}))=h^{0}(\mathcal{O}_{\Sigma_{n}}(-\sigma))=0$. Therefore 
\begin{align*}
    h^{1}(\mathcal{I}_{Z}(\sigma+K_{\Sigma_{n}}))=-\chi(\mathcal{I}_{Z}(\sigma+K_{\Sigma_{n}}))=|Z|-\chi(\mathcal{O}_{\Sigma_{n}}(\sigma+K_{\Sigma_{n}}))=|Z|=c_{2}.
\end{align*}
Since $\text{ext}^{1}(\mathcal{I}_{Z}(\sigma), \mathcal{O}_{\Sigma_{n}})=|Z|=c_{2}$ is constant on $Hilb^{c_{2}}(\Sigma_{n})$, by Grauert theorem (see \cite[Corollary 3.12.9]{Hartshorne}), we get $\mathcal{V}:=R^{1}{\pi_{1}}_{*}(\mathcal{F})$ is locally free on $Hilb^{c_{2}}(\Sigma_{n})$ and for every $Z\in Hilb^{c_{2}}(\Sigma_{n})$ the natural map $\mathcal{V}\otimes\mathbb{C}(Z)\longrightarrow H^{1}(X_{Z},\mathcal{F}_{Z})$ is an isomorphism; that is, for every $Z\in Hilb^{c_{2}}(\Sigma_{n})$ we have that $\mathcal{V}(Z)\cong Ext^{1}(\mathcal{I}_{Z}(\sigma),\mathcal{O}_{\Sigma_{n}})$, then $S=\mathbb{P}(\mathcal{V})$. Furthermore, since $\mathcal{V}$ is a locally free bundle of rank $c_{2}$ on $Hilb^{c_{2}}(\Sigma_{n})$, its geometric projectivization is, by definition, a projective bundle whose fibers are
\begin{align*}
    S_{Z}=\mathbb{P}(\mathcal{V}(Z))\cong\mathbb{P}(Ext^{1}(\mathcal{I}_{Z}(\sigma),\mathcal{O}_{\Sigma_{n}}))\cong\mathbb{P}^{c_{2}-1}.
\end{align*}
So, $S$ is a $({c_2-1})$-projective bundle on the Hilbert scheme $Hilb^{c_{2}}(\Sigma_{n})$, which parametrizes $0$-dimensional subschemes $Z$ of length $|Z|=c_{2}$. Since $Hilb^{c_{2}}(\Sigma_{n})$ is irreducible, then $S$ is also irreducible. Moreover, since $\text{dim}(Hilb^{c_{2}}(\Sigma_{n}))=2c_{2}$, there are $2c_{2}$ ways to choose $Z$ and $c_{2}-1$ ways to choose the extension, so we have $2c_{2}+c_{2}-1=3c_{2}-1$ ways to choose a projective bundle, therefore $\text{dim}(S)=\rho_{H_{r}}^{1}(2;\sigma,c_{2})=3c_{2}-1$; and there is a modular map $\pi$ from an open dense subset $S'$ of $S$ to $W_{H_{t}}^{1}(2;\sigma,c_{2})\setminus W_{H_{t}}^{2}(2;\sigma,c_{2})$.

Let us take the modular map $\pi:S'\longrightarrow W_{H_{t}}^{1}(2;\sigma,c_{2})\setminus W_{H_{t}}^{2}(2;\sigma,c_{2})$ such that it assigns to each extension the isomorphism class of the resulting vector bundle $E$ of the moduli space. Let $W\subset W_{H_{t}}^{1}(2;\sigma,c_{2})\setminus W_{H_{t}}^{2}(2;\sigma,c_{2})$ be the irreducible component containing the image of $S'$ by $\pi$. Since a general point of $\text{Im}(\pi)$ is a smooth point of $W$, then $W$ has expected dimension $3c_{2}-1$. Since $\text{dim}(W)=\text{dim}(S)$ then $\pi$ is generically finite over an open set $W'$ of $W$; that is, a general point $E\in W'$ has finite fiber. Then $S$ dominates $W$; that is, $\text{Im}(\pi)$ is dense in $W$. Since any isomorphism class $[E]\in W_{H_{t}}^{1}(2;\sigma,c_{2})\setminus W_{H_{t}}^{2}(2;\sigma,c_{2})$ arises as a nontrivial extension of type \ref{squence 4.3} for some $Z\in Hilb^{c_{2}}(\Sigma_{n})$, then $S$ finitely dominates $W_{H_{t}}^{1}(2;\sigma,c_{2})\setminus W_{H_{t}}^{2}(2;\sigma,c_{2})$. Since $W$ is an irreducible component of $W_{H_{t}}^{1}(2;\sigma,c_{2})\setminus W_{H_{t}}^{2}(2;\sigma,c_{2})$ and $S$ dominates both, we have $W=W_{H_{t}}^{1}(2;\sigma,c_{2})\setminus W_{H_{t}}^{2}(2;\sigma,c_{2})$, and therefore, $W_{H_{t}}^{1}(2;\sigma,c_{2})\setminus W_{H_{t}}^{2}(2;\sigma,c_{2})$ is irreducible.

The proof of the irreducibility of the varieties $W_{H}^{k}(2;c_{1},c_{2})\setminus W_{H}^{k+1}(2;c_{1},c_{2})$ in several subsequent cases will be similar, and therefore can be omitted from the respective proofs. 
\end{proof}

The irreducibility test in the previous proposition is based entirely on the same thechnique that appears in \cite[Proposition 4.12]{Costa3} of L. Costa and I. Macías-Tarrío.

\begin{proposition}\label{Proposition 4.7m}
Let $c_{2}=1$, $\Sigma_{0}$ be a Hirzebruch surface and $H_{t}\equiv \alpha\sigma+tf$ an ample divisor on $\Sigma_{0}$ with $0<t\leq\alpha$. Then, $W_{H_{t}}^{2}(2;\sigma,1)\setminus W_{H_{t}}^{3}(2;\sigma,1)$ is smooth and irreducible of the expected dimension, namely $\rho_{H_{t}}^{2}=1$.
\end{proposition}

\begin{proof}
Analogously to Proposition \ref{Proposition 4.8}, any vector bundle $E\in W_{H_{t}}^{2}(2;\sigma,1)\setminus W_{H_{t}}^{3}(2;\sigma,1)$ fits into a short exact sequence of type
\begin{align*}
    0 \longrightarrow \mathcal{O}_{\Sigma_{0}}\longrightarrow E\longrightarrow \mathcal{I}_{Z}(\sigma)\longrightarrow 0,
\end{align*}
and also $E$ is a smooth point of $M_{\Sigma_{0},H_{t}}(2;\sigma,1)$.

Now, note that by Serre duality, Proposition \ref{Proposition 4.3.1}, and Proposition \ref{Proposition 2.4}, we have that $H^{1}(E^{*}\otimes\mathcal{O}_{\Sigma_{0}}(K_{\Sigma_{0}}))=\{0\}$; thus, $\mu_{E}$ is trivially injective. Therefore, by Corollary \ref{Corollary 4.1}, $W_{H_{t}}^{2}(2;\sigma,1)\setminus W_{H_{t}}^{3}(2;\sigma,1)$ is smooth and of the expected dimension at $E$, which is $\rho_{H_{t}}^{1}(2;\sigma,1)=1$ for all $E\in W_{H_{t}}^{2}(2;\sigma,1)\setminus W_{H_{t}}^{3}(2;\sigma,1)$. Hence, $W_{H_{t}}^{2}(2;\sigma,1)\setminus W_{H_{t}}^{3}(2;\sigma,1)$ is smooth and of the expected dimension.

It only remains to prove the irreducibility of $W_{H_{t}}^{2}(2;\sigma,1)\setminus W_{H_{t}}^{3}(2;\sigma,1)$. Note that, by Proposition \ref{Proposition 4.4}, $M_{\Sigma_{0},H_{t}}(2;\sigma,1)$ is a smooth irreducible variety of dimension one. Then by Corollary \ref{Corollary 4.5}, $\text{dim}(W_{H_{t}}^{2}(2;\sigma,1))=1$, and since $M_{\Sigma_{0},H_{t}}(2;\sigma,1)$ is irreducible, then necessarily $M_{\Sigma_{0},H_{t}}(2;\sigma,1)=W_{H_{t}}^{2}(2;\sigma,1))$. Similarly, since $W_{H_{t}}^{2}(2;\sigma,1)\setminus W_{H_{t}}^{3}(2;\sigma,1)$ is irreducible, then $M_{\Sigma_{0},H_{t}}(2;\sigma,1)=W_{H_{t}}^{2}(2;\sigma,1)\setminus W_{H_{t}}^{3}(2;\sigma,1)$.
\end{proof}

\begin{proposition}\label{Proposition 4.15}
Let $c_{2}\geq 2$, $\Sigma_{n}$ be a Hirzebruch surface and $H\equiv \alpha\sigma+\beta f$ an ample divisor on $\Sigma_{n}$ with $\beta>\alpha n+(\alpha-1)$. Then $W_{H}^{1}(2;f,c_{2})\setminus W_{H}^{2}(2;f,c_{2})$ is smooth and irreducible of the expected dimension, namely $\rho_{H}^{1}=3c_{2}-1$.
\end{proposition}

\begin{proof}
The idea behind the proof is analogous to Proposition \ref{Proposition 4.8}. To prove that the divisor $D$ initiating the extension of $E$ is trivial, consider the cases $\alpha=1$ and $\alpha>1$.
\end{proof}

\begin{proposition}
Let $c_{2}=1$ and $\Sigma_{n}$ be a Hirzebruch surface and $H\equiv \alpha\sigma+\beta f$ an ample divisor on $\Sigma_{n}$ with $\beta>\alpha n+(\alpha-1)$. Then, $W_{H}^{2}(2;f,1)\setminus W_{H}^{3}(2;f,1)$ is smooth and irreducible of the expected dimension, namely $\rho_{H}^{2}=1$.
\end{proposition}

\begin{proof}
The proof of the injectivity of $\mu_{E}$ is similar to the proof of Proposition \ref{Proposition 4.7m}. The rest of the proof is analogous to Proposition \ref{Proposition 4.8}.
\end{proof}

\begin{proposition}\label{Proposition 4.12}
Let $c_{2}=2$ and $\Sigma_{0}$ be a Hirzebruch surface and $H_{r}\equiv \alpha\sigma+\beta_{r} f$ an ample divisor on $\Sigma_{0}$ where $\beta_{r}=\alpha+r$ with $0<r\leq\alpha$. Then, $W_{H_{r}}^{3}(2;\sigma+f,2)\setminus W_{H_{r}}^{4}(2;\sigma+f,2)$ is smooth and irreducible of the expected dimension $\rho_{H_{r}}^{3}=3$.
\end{proposition}

\begin{proof}
Any vector bundle $E\in W_{H_{r}}^{3}(2;\sigma+f,2)\setminus W_{H_{r}}^{4}(2;\sigma+f,2)$ sits in a short exact sequence 
\begin{align}
    0 \longrightarrow \mathcal{O}_{\Sigma_{0}}(f)\longrightarrow E\longrightarrow \mathcal{I}_{Z}(\sigma)\longrightarrow 0, \label{sequence 4.8.1}
\end{align}
where $Z$ is a locally complete intersection $0$-cycle of length $|Z|=1$ or in a short exact sequence 
\begin{align}
    0 \longrightarrow \mathcal{O}_{\Sigma_{0}}\longrightarrow E\longrightarrow \mathcal{I}_{Z}(\sigma+f)\longrightarrow 0, \label{sequence 4.8.2}
\end{align}
where $|Z|=2$.

Note that, in this last case, $E$ should belong to $M_{H_{r}}(2;\sigma+f,2)$. However, we have not proved that a vector bundle $E$ arising from a short exact sequence of the above form is indeed $H_{r}$-stable. We only have such a result for extensions of this type when $c_{2}\geq 5$ (see Theorem \ref{Theorem 3.4}-(6)), and therefore, we exclude this case.

Assume that $E$ fits into a short exact sequence of type \ref{sequence 4.8.1}. The proof of the injectivity of $\mu_{E}$ is similar to the proof of Proposition \ref{Proposition 4.7m}. The rest of the proof is analogous to Proposition \ref{Proposition 4.8}.
\end{proof}

\begin{proposition}
Let $c_{2}=5$ and $\Sigma_{0}$ be a Hirzebruch surface and $H_{r}\equiv \alpha\sigma+\beta_{r} f$ an ample divisor on $\Sigma_{0}$ where $\beta_{r}=\alpha+r$ with $0<r\leq\alpha$. Then, $W_{H_{r}}^{2}(2;\sigma+f,5)\setminus W_{H_{r}}^{3}(2;\sigma+f,5)$ is smooth and irreducible of the expected dimension $\rho_{H_{r}}^{2}=11$.
\end{proposition}

\begin{proof}
Any vector bundle $E\in W_{H_{r}}^{2}(2;\sigma+f,5)\setminus W_{H_{r}}^{3}(2;\sigma+f,5)$ sits in a short exact sequence of type
\begin{align}
    0 \longrightarrow \mathcal{O}_{\Sigma_{0}}(f)\longrightarrow E\longrightarrow \mathcal{I}_{Z}(\sigma)\longrightarrow 0, \label{Example 4.13.4.8}
\end{align}
where $Z$ is a locally complete intersection $0$-cycle of length $|Z|=4$ or in a short exact sequence
\begin{align}
    0 \longrightarrow \mathcal{O}_{\Sigma_{0}}\longrightarrow E\longrightarrow \mathcal{I}_{Z}(\sigma+f)\longrightarrow 0, \label{Example 4.13.4.9}
\end{align}
where $|Z|=5$.

Suppose that $E$ fits into a short exact sequence of type \ref{Example 4.13.4.9}. Since $h^{0}(\mathcal{O}_{\Sigma_{0}}(\sigma+f))=4$ and $|Z|=5$, we have that $h^{0}(\mathcal{I}_{Z}(\sigma+f))=0$ by Proposition \ref{Proposition 1.27}, and then $h^{0}(E)=1$ which is impossible since $E\in W_{H_{r}}^{2}(2;\sigma+f,5)\setminus W_{H_{r}}^{3}(2;\sigma+f,5)$. Then $E$ fits into a short exact sequence of the form \ref{Example 4.13.4.8}.

By Theorem \ref{Theorem 4.6}, the map $\mu_{E}$ is injective because $h^{1}(\mathcal{O}_{\Sigma_{0}}(\sigma+f))=0$. The rest of the proof is similar to Proposition \ref{Proposition 4.8}.
\end{proof}

\begin{proposition}
Let $c_{2}\geq 2$. If $(c_{2},n)\in\{(c_{2},1),(3,2)\}$, $\Sigma_{n}$ be a Hirzebruch surface and $H_{r}\equiv \alpha\sigma+\beta_{r} f$ an ample divisor on $\Sigma_{n}$ where $\beta_{r}=\alpha+r$ with $0<r\leq \alpha$. Then, $W_{H_{r}}^{2}(2;\sigma+f,c_{2})\setminus W_{H_{r}}^{3}(2;\sigma+f,c_{2})$ is smooth and irreducible of the expected dimension $\rho_{H_{r}}^{2}=2c_{2}-n+1$.
\end{proposition}

\begin{proof}
Analogously to Proposition \ref{Proposition 4.12}, we will assume that $E\in W_{H_{r}}^{2}(2;\sigma+f,c_{2})\setminus W_{H_{r}}^{3}(2;\sigma+f,c_{2})$ sits in a short exact sequence of type
\begin{align}
    0 \longrightarrow \mathcal{O}_{\Sigma_{n}}(D)\longrightarrow E\longrightarrow \mathcal{I}_{Z}(\sigma+f-D)\longrightarrow 0, \label{Example 4.4.14}
\end{align}
where $D=f$ and $Z$ is a locally complete intersection $0$-cycle of length $|Z|=c_{2}-1$, since we do not have stability results for $H_{r}$ when $D=0$ and $c_{2}\in\{2,3,4\}$.

By Theorem \ref{Theorem 4.6}, the map $\mu_{E}$ is injective if $c_{2}=3$ and $n=2$ because $h^{1}(\mathcal{O}_{\Sigma_{2}}(\sigma+f))=0$.

On the other hand, note that by Serre duality we have $H^{1}(E^{*}\otimes\mathcal{O}_{\Sigma_{n}}(K_{\Sigma_{n}}))=H^{1}(E)$ and since $H^{i}(\mathcal{O}_{\Sigma_{n}}(f))=0$ for $i\in\{1,2\}$, then  $h^{1}(E)=h^{1}(\mathcal{I}_{Z}(\sigma))=h^{1}(\mathcal{O}_{\Sigma_{n}}(K_{\Sigma_{n}}-\sigma))=n-1$, so if $n=1$, then $H^{1}(E^{*}\otimes\mathcal{O}_{\Sigma_{n}}(K_{\Sigma_{n}}))=0$, and therefore $\mu_{E}$ is trivially injective for every value $c_{2}\geq 2$.

The rest of the proof is analogous to Proposition \ref{Proposition 4.8}.
\end{proof}

\begin{proposition}
Let $c_{2}\geq 5$, $\Sigma_{n}$ be a Hirzebruch surface and $H_{r}\equiv \alpha\sigma+\beta_{r} f$ an ample divisor on $\Sigma_{n}$ where $\beta_{r}=\alpha(n+1)+r$ with $0<r\leq\alpha$. Then, $W_{H_{r}}^{1}(2;\sigma+f,c_{2})\setminus W_{H_{r}}^{2}(2;\sigma+f,c_{2})$ is smooth and irreducible of the expected dimension $\rho_{H_{r}}^{1}=3c_{2}-1$.
\end{proposition}

\begin{proof}
The proof is analogous to Proposition \ref{Proposition 4.8}.
\end{proof}

\vspace{1em}

{\bf Acknowledgements}

The author would like to thank the Centro de Ciencias Matemáticas at Universidad Nacional Autónoma de México (UNAM) for providing the facilities and support necessary for the completion of this work. The author also gratefully acknowledges the Secretaría de Ciencia, Humanidades, Tecnología e Innovación (SECIHTI) of Mexico for financial support received through a doctoral scholarship. Additionally, the author was partially supported by UNAM-PAPIIT grant IN101226, ``Teoría de Brill-Noether y aplicaciones''.



\begin{thebibliography}{99}

\bibitem{Beauville} \textsc{A. Beauville}, \textit{Complex Algebraic Surfaces}, London Mathematical Society, Student Texts. 34, Cambridge University Press, (1996).

\bibitem{Coskun} \textsc{I. Coskun}, \textit{Brill-Noether theorems and globally generated vector bundles on Hirzebruch surfaces}, Nagoya Math. J., 238 (2020), 1-36.

\bibitem{Costa3} \textsc{L. Costa and I. Macías-Tarrío}, \textit{Brill-Noether theory of stable vector bundles on ruled surfaces}, Mediterr. J. Math. (2024) 21:118.

\bibitem{Costa4} \textsc{L. Costa and R. M. Miró-Roig}, \textit{Brill-Noether theory for moduli spaces of sheaves on algebraic varieties}, Forum Math. 22 (2010), 411–432.

\bibitem{Costa2} \textsc{L. Costa and R. M. Miró-Roig}, \textit{Brill–Noether theory on Hirzebruch surfaces}, Journal of Pure and Applied Algebra, Vol. 214, Issue 9, (2010), 1612-1622.

\bibitem{Filimon} \textsc{L. Filimon}, \textit{On translated rank-2 Brill–Noether loci on regular surfaces}, Arch. Math. 118(3), 271–281 (2022).

\bibitem{Friedman} \textsc{R. Friedman}, \textit{Algebraic Surfaces and Holomorphic Vector Bundles}, Springer Verlag New York, Inc. (1998).

\bibitem{Hartshorne} \textsc{R. Hartshorne}, \textit{Algebraic Geometry}, G.T.M., vol. 52, Springer, Berlin (1977).

\bibitem{Huybrechts} \textsc{D. Huybrechts, M. Lehn}, \textit{The Geometry of Moduli Spaces of Sheaves}, Cambridge University Press, (2010).

\bibitem{Maruyama} \textsc{M.Maruyama}, \textit{Openness of a family of torsion free sheaves}, J. Math. Kyoto Univ. 16, (1976), 627-637.

\bibitem{Roa} \textsc{L. Roa-Lequizam\'on, H. Torres-López, A. G. Zamora}, \textit{On the Segre invariant for rank two vector bundles on $\mathbb{P}^{2}$}. Adv. Geom. 2021; 21(4):565-576.

\bibitem{Timofeeva} \textsc{N.V. Timofeeva}, \textit{Determinantal Resolution of the Universal Subscheme in $S\times H_{d+1}$}. Mathematical Notes, vol. 69, no. 2, 2001, pp. 253-261.

\end{thebibliography}
\end{document}